\documentclass[12pt]{amsart}
\usepackage{amssymb,bm,graphicx,graphics,mathrsfs,bbm,color,float,accents}
\usepackage{subcaption}
\usepackage{longtable}
\numberwithin{equation}{section}
\numberwithin{figure}{section}
\numberwithin{table}{section}
\allowdisplaybreaks[4]
\theoremstyle{plain}
\newtheorem{theorem}{Theorem}[section]
\newtheorem{lemma}[theorem]{Lemma}
\newtheorem{corollary}[theorem]{Corollary}
\theoremstyle{remark}
\newtheorem{remark}[theorem]{Remark}
\begin{document}
\def\O{\Omega}
\def\G{\Gamma}
\def\E{\mathbb{E}}
\def\R{\mathbb{R}}
\def\p{\partial}
\def\cA{\mathcal{A}}
\def\cT{\mathcal{T}}
\def\cN{\mathcal{N}}
\def\Th{\cT_h}
\def\ED{\mathcal{E}_D}
\def\LT{{L^2(\O)}}
\def\LTD{{L^2(D)}}
\def\LTT{[\LT]^2}
\def\LTz{L^2_0(\O)}
\def\Poly{\mathbb{P}}
\def\curl{\mathrm{curl}\,}
\def\bcurl{\mathbf{curl}\,}
\def\grad{\mathbf{grad}\,}
\def\div{\mathrm{div}\,}
\def\bu{\bm{u}}
\def\bv{\bm{v}}
\def\bn{\bm{n}}
\def\bof{\bm{f}}
\def\Hcurl{H(\mathrm{curl};\O)}
\def\HzCurl{H_0(\mathrm{curl};\O)}
\def\HDivz{H(\mathrm{div}^0;\O)}
\def\HDiv{H(\mathrm{div};\O)}
\def\HOne{{H^1(\O)}}
\def\HOneD{{H^1(D)}}
\def\HOnez{H^1_0(\O)} \def\HOneZ{H^1(\O)\cap\LTz}
\def\zV{\mathring{V}}
\def\Cc{C_{\rm coer}}
\def\hfactor{h^{\min(\ExpoStar,k)}}
\def\hfactorG{h^{\min(\Expo,k)}}
\def\hfactorE{h^{\min(\ExpoE,k)}}
\def\PiO{\Pi_{k,h}^1}
\def\PiZ{\Pi_{k,h}^0}
\def\Id{\mathrm{Id}}
\def\bA{\mathbf{A}}
\def\bc{\mathbf{c}}
\def\bb{\mathbf{b}}
\def\red{\color{red}}
\def\cQ{\mathcal{Q}}
\def\Expo{(\pi/\omega)}
\def\ExpoE{(\pi/\omega)-\epsilon}
\def\Reg{{H^{1+\Expo}(\O)}}
\def\RegE{{H^{1+\ExpoE}(\O)}}
\def\ExpoStar{(\pi/\omega)^*}
\def\RegStar{{H^{1+\ExpoStar}(\O)}}
\def\CExpo{C_{\raise 1pt\hbox{$\scriptscriptstyle\O$}}^*}
\def\CO{C_\Omega}
\def\COE{C_{\Omega,\epsilon}}
\def\eBdry{e_h^{\scriptscriptstyle\p\O}}
%
\title[Virtual Element Methods for a Planar Quad-Curl Problem]
{Virtual element methods for a quad-curl problem on general planar domains}
\thanks{This  work  was supported in part by
the National Science Foundation under
 Grant No. DMS-25-13273.}
%
\author{Susanne C. Brenner} \address{Susanne C. Brenner,
Department of Mathematics and Center for
Computation and Technology, Louisiana State University, Baton Rouge,
LA 70803, USA}%
\email{brenner@math.lsu.edu}
\author{Li-yeng Sung} \address{Li-yeng Sung,
 Department of Mathematics and Center for Computation and Technology,
 Louisiana State University, Baton Rouge, LA 70803, USA}%
\email{sung@math.lsu.edu}
\author{Jai Tushar}
\address{Jai Tushar, Department of Mathematics and Center for
Computation and Technology, Louisiana State University, Baton Rouge,
LA 70803, USA}
\email{Jai.Tushar@lsu.edu}
\begin{abstract}
  We  design and analyze virtual element methods for a
  quad-curl problem on general polygonal domains that are
  based on the Hodge decomposition of divergence-free vector fields.
  Numerical results that corroborate the theoretical analysis are also presented.
\end{abstract}
\date{April 20, 2026}
\maketitle
\section{Introduction}\label{sec:Introduction}
 Let $\O$ be a bounded polygonal domain in $\R^2$, $\bof\in\LTT$
 and $\beta,\gamma$ be nonnegative constants,
 where we assume $\gamma>0$ if $\O$ is not simply connected.
 The quad-curl problem is to find
 $\bu\in \E$ such that
\begin{equation}     \label{eq:QuadCurl}
 \big(\bcurl(\curl\bu),\bcurl(\curl\bv)\big)_\LT
 +\beta(\curl\bu,\curl\bv)_\LT+\gamma(\bu,\bv)_\LT=
(\bof,\bv)_\LT
\end{equation}
 for all $\bv\in\E$, where the energy space $\E$ is defined by
\begin{equation*}
 \E=\{\bv\in\LTT:\,\curl\bv\in H^1_0(\O), \;\div\bv=0
\quad\text{and}\quad \bn\times\bv=0\quad\text{on}\;\p\O\},
\end{equation*}
 and $\bn$ is the unit outer normal on $\p\O$. $\E$ is a Hilbert space under the inner product
\begin{equation*}
  (\bu,\bv)_\E=(\bu,\bv)_\LT+\big(\bcurl(\curl\bu),\bcurl(\curl\bv)\big)_\LT.
\end{equation*}
 It was shown in \cite{BSS:2017:Curl4} that \eqref{eq:QuadCurl} is uniquely solvable.
\begin{remark} \label{rem:Notation}
  We follow standard notation for function spaces, norms and differential
   operators that can be found for example in
  \cite{ADAMS:2003:Sobolev,Ciarlet:1978:FEM,Monk:2003:Maxwell,BScott:2008:FEM,BGHRR:2025:NS}.
\end{remark}
\begin{remark}\label{rem:Curl}
For a vector field $\bv=[v_1,v_2]^t$,   the curl of $\bv$ is a scalar function given by
\begin{equation*}
  \curl\bv=(\p v_2/\p x_1)-(\p v_1/\p x_2). \end{equation*}
 For a scalar function $\phi$, the $\mathbf{curl}$ of $\phi$ is a vector field given by
\begin{equation*}
  \bcurl\phi=[\p \phi/\p x_2,-\p \phi/\p x_1]^t.
\end{equation*}
\end{remark}
\par
 The quad-curl problem \eqref{eq:QuadCurl} is related to the Maxwell
 transmission eigenvalue problem
 \cite{CH:2007:Transmission,CCMS:2010:Anisotropic,MS:2012:MTE}
 and magnetohydrodynamics with hyperresistivity
 (cf. \cite{Biskamp:2005:Plasmas,CSZ:2007:EMHD}).
 There is a growing literature on numerical solutions of \eqref{eq:QuadCurl} and its
 three dimensional version
 (cf. \cite{ZHX:2011:Curl4,HHSX:2012:Curl4,Sun:2016:Curl4,BSS:2017:Curl4,SCGW:2018:QuadCurl,
 SZZ:2019:WGQuadCurl,BCS:2019:MGHodge,
 ZWZ:2019:QuadCurl,HZZ:2020:CurlCurlConforming,ZZ:2021:VEMQuadCurl,CQX:2021:QuadCurl,
 WSLZ:2021:QuardCurlSpectral,CCH:2022:QuarCurl,WWZ:2023:WeakGalerkinQuadCurl,
 ZZ:2023:NonconformingQuadCurl,Huang:2023:QuadCurl,FHH:2023:QuadCurl,BCS:2024:3DQuadCurl,
 HZ:2024:QuadCurl,WCS:2024:QuadCurl,
 WLZZ:2025:QuadCurl,GZZZ:2025:QuadCurl,CHZ:2025:QuadCurl,LMWZ:2025:QuadCurl,ZDWZ:2026:QuadCurl,
 WZ:2026:QuadCurl,WZZ:2026:QuadCurl}).
 Among these references,
 there is only one polytopal method for the two dimensional quad-curl problem
 in \cite{ZZ:2021:VEMQuadCurl} that extends the
 $H(\curl^2)$ conforming finite element methods in \cite{HZZ:2020:CurlCurlConforming}
 to the virtual element setting.
 The goal of this paper is to design and analyze virtual element methods
 for \eqref{eq:QuadCurl} that are based on the Hodge decomposition
 approach in \cite{BSS:2017:Curl4},
 which allows one to use standard virtual element methods for second order elliptic
 boundary value problems
 (cf. \cite{BBCMMR:2013:VEM,AABMR:2013:Projector,BLR:2017:VEM,
 BGS:2017:VEM2,BSung:2018:VEM,BBMR:2023:VEMActa}),
 and consequently
 the numerical schemes in this paper are simpler than the ones in
 \cite{ZZ:2021:VEMQuadCurl}.
\par
 The rest of the paper is organized as follows.  We recall the Hodge
 decomposition approach to \eqref{eq:QuadCurl}
 in Section~\ref{sec:Hodge} and present the virtual element methods
 in Section~\ref{sec:VEM}.
 The error analysis is carried out in Section~\ref{sec:xihError} and
 Section~\ref{sec:uhError}.
 We report some numerical results in Section~\ref{sec:Numerics} and end
 with some concluding remarks
 in Section~\ref{sec:Conclusions}.  For the convenience of the readers,
 we also provide in
 Appendix~\ref{append:VEMError} the error analysis
 of the virtual element methods for the Dirichlet and Neumann boundary value problems
 that appear in this paper.
\par
 Throughout the paper  we use $C$ (with or without subscripts) to
 denote a generic positive constant
 that is independent of the mesh size.
 \par
 We record here two Poincar\'e-Friedrichs inequalities
 (cf. \cite{ADAMS:2003:Sobolev,Necas:2012:Direct}):
\begin{alignat} {3}
\|v\|_\LT&\leq C_\O\big[\big|(v,1)_\LT\big|+  |v|_\HOne\big]&\qquad&\forall\,v\in \HOne,
  \label{eq:PFNeumann}    \\
\|v\|_\LT&\leq C_\O |v|_\HOne &\qquad&\forall\,v\in \HOnez,
\label{eq:PFDirichlet}
\end{alignat}
 that will be used frequently in the
 analysis of the Hodge decomposition method.
\par
 The following integration by parts formulas
  (cf. \cite[Section~1.2.4.1 and Section~1.2.4.2]{BGHRR:2025:NS}) are also useful:
\begin{equation} \label{eq:IBVP1}
  (\bv,\bcurl \psi)_\LT=(\curl\bv,\psi)_\LT+
  \int_{\p\O}(\bv\times\bn)\psi\, ds
\end{equation}
 for all $\bv\in\Hcurl$ and $\psi\in\HOne$, and
\begin{equation} \label{eq:IBVP2}
  (\bv,\grad\psi)_\LT=-(\div\bv,\psi)_\LT+
  \int_{\p\O}(\bv\cdot\bn)\psi\,ds
\end{equation}
 for all $\bv\in \HDiv$ and $\psi\in\HOne$.

\section{The Hodge Decomposition Approach}\label{sec:Hodge}
 Let $m$ be the Betti number of $\O$ (cf. Figure~\ref{fig:Betti}
 for examples with $m=0$, $1$ and $2$).
\begin{figure}[H]
\begin{center}
  \includegraphics[height=1.2in]{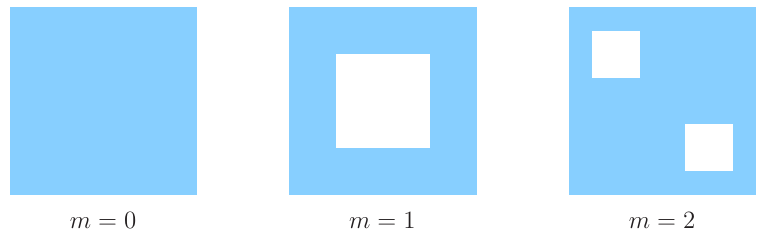}
\end{center}
\caption{Domains with Betti numbers $0$, $1$ and $2$}
\label{fig:Betti}
\end{figure}
\par
 It follows from the Hodge decomposition of divergence-free vector fields
 (cf. \cite{BCNS:2012:Hodge} and \cite[Section~2.2.1]{BGHRR:2025:NS})
 that we can represent the solution
 $\bu$ of \eqref{eq:QuadCurl} in the following form:
\begin{equation}\label{eq:HodgeDecomposition}
  \bu=\bcurl\phi+\sum_{j=1}^m c_j\grad\varphi_j,
\end{equation}
 where $c_1,\ldots,c_m\in\R$,   $\phi\in \HOneZ$,
$\LTz=\{v\in\LT: (v,1)_\LT=0\}$, and the harmonic functions
$\varphi_1,\ldots,\varphi_m$ are defined as follows.
\par
 Let $\G_0$ be the outer boundary of $\O$ and $\Gamma_1,\ldots,\G_m$
 be the components of the inner boundary of $\O$.
 The harmonic functions $\varphi_1,\ldots,\varphi_m\in \HOne$  are determined by
\begin{subequations}\label{subeqs:HarmonicFunctions}
\begin{alignat}{3}
  (\grad\varphi_j,\grad v)_\LT&=0&\qquad&\forall\,v\in \HOnez,
  \label{eq:Harmonic}   \\
  \varphi_j\Big|_{\G_0}&=0,\label{eq:OuterBoundary}\\
  \varphi_j\Big|_{\G_k}&=\begin{cases}
    1 &\qquad\text{if $k=j$}\\[4pt]
    0&\qquad\text{if $k\neq j$}
  \end{cases}        &\qquad&\text{for}\;1\leq j,k\leq m.
  \label{eq:InnerBoundary}
\end{alignat}
\end{subequations}
\par
  Let $\omega$ be the largest interior angle of $\O$ and
   $\epsilon$ be an arbitrary positive number.
  The harmonic functions $\varphi_j$ for $1\leq j\leq m$
  satisfy the estimates
\begin{subequations}\label{subeqs:HarmonicRegularity}
\begin{alignat}{3}
  \|\varphi_j\|_\Reg&\leq \CO
  &\qquad&\text{if $(\pi/\omega)$ is not an integer},\label{eq:Generic}\\
  \|\varphi_j\|_\RegE&\leq \COE  &\qquad&\text{if $(\pi/\omega)$
   is an integer},\label{eq:Exceprional}
\end{alignat}
\end{subequations}
 by elliptic regularity (cf. \cite[Section~5.1]{Grisvard:1985:EPN} and
 \cite[Section~2.5]{Dauge:1988:EBV}).
\begin{remark}\label{rem:RegularityNotation}
  For simplicity, we will use $\ExpoStar$ to denote $(\pi/\omega)$
   and $\CExpo$ to denote $\CO$ if $(\pi/\omega)$ is not an integer,
   and we will use
   $\ExpoStar$ to denote $(\pi/\omega)-\epsilon$ and
   $\CExpo$ to denote $\COE$ if $(\pi/\omega)$ is an integer.
   Consequently we can rewrite
   \eqref{subeqs:HarmonicRegularity} concisely as
\begin{equation*}
  \|\varphi_j\|_\RegStar\leq \CExpo.
\end{equation*}
\end{remark}
\begin{remark}\label{rem:Orthogonality}
  It follows from \eqref{eq:IBVP1}, \eqref{eq:OuterBoundary}
  and \eqref{eq:InnerBoundary}  that
\begin{equation*}
  (\bcurl\psi,\grad\varphi_j)_\LT=0 \qquad\text{for}\;\psi\in
  \HOne \quad\text{and}\quad 1\leq j\leq m
\end{equation*}
(cf. \cite[Lemma~2.4]{BCNS:2012:Hodge}).
\end{remark}
\begin{remark}   \label{rem:HarmonicFunctions}
In view of \eqref{subeqs:HarmonicFunctions},
  the vector fields $\grad\varphi_1,\ldots,\grad\varphi_m$
  belong to the energy space $\E$
  (cf. \cite[Corollary~2.5]{BCNS:2012:Hodge} and
  \cite[Chapter I, Remark~2.3]{GR:1986:NS}).
\end{remark}
\begin{remark}\label{rem:phiBoundary}
 Note that the definition of $\E$, the Hodge decomposition \eqref{eq:HodgeDecomposition}
  and the boundary conditions \eqref{eq:OuterBoundary} and \eqref{eq:InnerBoundary} imply
 \begin{equation*}
   0=\bn\times\bu=\bn\times\Big(\bcurl\phi+
   \sum_{j=1}^m c_j\grad\varphi_j\Big)=\bn\times\bcurl\phi.
\end{equation*}
\end{remark}
\begin{remark}\label{rem:Convention}
  We follow the convention that the summation in \eqref{eq:HodgeDecomposition}
  is ignored if $m=0$.
\end{remark}
\par
 We find $\bu$
 through \eqref{eq:HodgeDecomposition}  in the Hodge decomposition
 approach to \eqref{eq:QuadCurl}.
 Therefore we need to find $\phi$, and also
 $\varphi_1,\ldots,\varphi_m$ and $c_1,\ldots,c_m$ if $m\geq1$.
 Below we sketch the derivations of the relevant problems that determine these unknowns.
 Details can be found in
 \cite{BSS:2017:Curl4}.
%
\subsection{Problem for $\phi$}\label{subsec:phi}
 We can determine $\phi$ in \eqref{eq:HodgeDecomposition} through  the function
\begin{equation*}
  \xi=\curl\bu\in H^1_0(\O)\cap\LTz.
\end{equation*}
\par
 Indeed it is easy to see that \eqref{eq:IBVP1}, \eqref{eq:HodgeDecomposition}
  and Remark~\ref{rem:Orthogonality} imply
  $\phi\in\HOneZ$ satisfies
\begin{equation*}
  (\bcurl\phi,\bcurl\psi)_\LT=(\bu,\curl\psi)_\LT=(\xi,\psi)_\LT
  \qquad\forall\,\psi\in\HOneZ,
\end{equation*}
 or equivalently, the function $\phi\in\HOne$ satisfies
\begin{equation} \label{eq:phiEq2}
  (\bcurl\phi,\bcurl\psi)_\LT+(\phi,1)_\LT(\psi,1)_\LT=(\xi,\psi)_\LT
  \qquad\forall\,\psi\in\HOne,
\end{equation}
 which is a well-posed Neumann boundary value problem.
%
\subsection{Problem for $\xi$}\label{subsec:xi}
 Given any $\eta\in\HOnez\cap\LTz$, there exists a unique
 $\bv\in \HzCurl\cap \HDivz$ such that
\begin{equation}\label{eq:potential}
  \curl\bv=\eta.
\end{equation}
 In particular the vector field $\bv$ belongs to $\E$.
\par
 In fact, let $\theta\in \HOne$ be defined by
\begin{equation*}
  (\bcurl\theta,\bcurl\tau)_\LT+(\theta,1)_\LT(\tau,1)_\LT=(\eta,\tau)_\LT
  \qquad\forall\,\tau\in \HOne.
\end{equation*}
 Then $\bv=\bcurl\theta$ has the desired properties because of \eqref{eq:IBVP1}.
\par
 It follows from \eqref{eq:QuadCurl} and \eqref{eq:potential} that
\begin{equation}\label{eq:xiEquation1} (\bcurl\xi,\bcurl\eta)_\LT
+\beta(\xi,\eta)_\LT=(Q\bof-\gamma\bu,\bv)_\LT,
\end{equation}
 where $Q$ is the orthogonal projection from $\LTT$ onto $\HDivz$.
\par
 On the other hand, we have, for $1\leq j\leq m$,
\begin{equation}\label{eq:xiEquation2}
  \int_{\G_j}(Q\bof-\gamma\bu)\cdot \bn_{\G_j}ds
  =\int_\O (Q\bof-\gamma\bu)\cdot\grad\varphi_j dx=
  \int_\O (\bof-\gamma\bu)\cdot\grad\varphi_jdx=0 \end{equation}
 by \eqref{eq:QuadCurl}, \eqref{eq:IBVP2}, \eqref{eq:InnerBoundary},
 Remark~\ref{rem:Orthogonality} and Remark~\ref{rem:HarmonicFunctions}.
 Here $\bn_{\G_j}$ is the outer unit normal
 along $\G_j$.
\par
 Since $\mathrm{div}(Q\bof-\gamma\bu)=0$, we conclude from
 \eqref{eq:xiEquation2} and \cite[Lemma~2.2.1]{BGHRR:2025:NS} that there exists
 a unique $\rho\in \HOne\cap\LTz$ such that
\begin{equation}\label{eq:rho} \bcurl\rho=Q\bof-\gamma\bu.
\end{equation}
\par
 Putting \eqref{eq:IBVP1},  \eqref{eq:potential}, \eqref{eq:xiEquation1}
 and \eqref{eq:rho} together,
 we arrive at
\begin{equation}\label{eq:xiEquation}
  (\bcurl\xi,\bcurl\eta)_\LT+\beta(\xi,\eta)_\LT=(\rho,\eta)_\LT
  \qquad\forall\,\eta\in\HOnez\cap\LTz.
\end{equation}
\begin{remark}\label{rem:Simpler}
  The derivations of \eqref{eq:rho} and \eqref{eq:xiEquation} are simpler than the
  ones in \cite{BSS:2017:Curl4}.
\end{remark}
%
\subsubsection{The Case Where $\gamma=0$}\label{subsubsec:gammaZero}
 In  this case \eqref{eq:rho} and \eqref{eq:xiEquation}  are decoupled.
 We can rewrite \eqref{eq:rho} for $\rho\in \HOne\cap\LTz$ as
\begin{equation*}
  (\bcurl\rho,\bcurl \psi)_\LT=(\bof,\curl\psi)_\LT
  \qquad\forall\,\psi\in \HOne\cap\LTz,
\end{equation*}
 or equivalently, $\rho\in\HOne$ satisfies
\begin{equation}\label{eq:rhoEquation}
(\bcurl\rho,\bcurl\psi)_\LT
+(\rho,1)_\LT(\psi,1)_\LT=(\bof,\curl\psi)_\LT \qquad\forall\,\psi\in \HOne,
\end{equation}
 which is a well-posed Neumann boundary value problem.
\par
 Therefore one can first solve \eqref{eq:rhoEquation} for $\rho$ and
 then solve \eqref{eq:xiEquation}
 for $\xi$. Note that \eqref{eq:rhoEquation} implies
\begin{equation} \label{eq:rhoBdd}
  \|\bcurl\rho\|_\LT\leq \|\bof\|_\LT.
\end{equation}
\par
 Furthermore, the solution of the nonstandard boundary value
 problem \eqref{eq:xiEquation}
 can be reduced to  two standard well-posed Dirichlet
 boundary value problems:
\begin{alignat}{3}
  &\xi=\xi_0-[(1,\xi_0)_\LT/(1,\xi_1)_\LT]\xi_1,
  \label{eq:xiRepresentation}\\
  \intertext{where $\xi_0,\xi_1\in \HOnez$ are determined by}
  &(\bcurl\xi_0,\bcurl\eta)_\LT+\beta(\xi_0,\eta)_\LT=(\rho,\eta)_\LT
  &\qquad&\forall\,\eta\in\HOnez,\label{eq:xi0}\\
  &(\bcurl\xi_1,\bcurl\eta)_\LT+\beta(\xi_1,\eta)_\LT=(1,\eta)_\LT
    &\qquad&\forall\,\eta\in\HOnez.\label{eq:xi1}
\end{alignat}
\par
 Note that \eqref{eq:xi1} implies $(1,\xi_1)_\LT>0$ so that
 \eqref{eq:xiRepresentation} is well-defined.
\par
 Moreover it follows from  \eqref{eq:PFDirichlet},
 \eqref{eq:rhoBdd} and \eqref{eq:xi0} that
\begin{equation} \label{eq:xi0Bdd}
  \|\xi_0\|_\HOne\leq   C\|\bof\|_\LT.
\end{equation}
 \par
  Similarly it follows from \eqref{eq:PFDirichlet} and \eqref{eq:xi1} that
\begin{equation}\label{eq:xi1Bdd} \|\xi_1\|_\HOne\leq C.
\end{equation}
\par
 Combining \eqref{eq:xiRepresentation}, \eqref{eq:xi0Bdd}
 and \eqref{eq:xi1Bdd},
 we see that
\begin{equation}\label{eq:xiBddZero}
  \|\xi\|_\HOne\leq C\|\bof\|_\LT.
\end{equation}
%
\subsubsection{The Case Where $\gamma>0$}\label{subsubsec:gammaPositive}
  In this case  \eqref{eq:rho} and \eqref{eq:xiEquation} are
  coupled and we can write them
  concisely as the following system:
  Find $(\zeta,\xi)\in [\HOne\cap\LTz]\times[\HOnez\cap\LTz]$
  such that
\begin{equation}\label{eq:System}
 \cA\big((\zeta,\xi),(\psi,\eta)\big)=\gamma^{-\frac12}(\bof,\bcurl\psi)_\LT
\end{equation}
 for all $(\psi,\eta)\in [\HOne\cap\LTz]\times[\HOnez\cap\LTz]$, where
 $\zeta=\gamma^{-\frac12}\rho$ and the bilinear form   $\cA(\cdot,\cdot)$
 is defined by
\begin{align}\label{eq:cA}
  \cA\big((\zeta,\xi),(\psi,\eta)\big)&=(\bcurl\zeta,\bcurl\psi)_\LT
  +\gamma^\frac12(\psi,\xi)_\LT
  -\gamma^\frac12(\zeta,\eta)_\LT   \\
  &\hspace{50pt}+(\bcurl\xi,\bcurl\eta)_\LT+\beta(\xi,\eta)_\LT.\notag
\end{align}
\par
 Furthermore, the solution of \eqref{eq:System} can be reduced to the
 following systems posed on the space of
 functions without the zero mean constraint:
\begin{alignat}{3}
  &(\zeta,\xi)=(\zeta_0,\xi_0)-[(1,\xi_0)_\LT/(1,\xi_1)_\LT](\zeta_1,\xi_1),
  \label{eq:SystemRepresentation}\\
\intertext{where $(\zeta_0,\xi_0), (\zeta_1,\xi_1)\in \HOne\times\HOnez$
 are determined by}
 &\cA\big(\zeta_0,\xi_0),(\psi,\eta)\big)+(\zeta_0,1)_\LT(\psi,1)_\LT=
   \gamma^{-\frac12}(\bof,\bcurl\psi)_\LT, \label{eq:System0}\\
  &\cA\big(\zeta_1,\xi_1),(\psi,\eta)\big)+(\zeta_1,1)_\LT(\psi,1)_\LT =(1,\eta)_\LT,
  \label{eq:System1}
\end{alignat}
 for all $(\psi,\eta)\in \HOne\times\HOnez$.
\begin{remark}\label{rem:BoundedAndCoercive}
  The bilinear form on the left-hand side of \eqref{eq:System0}
  and \eqref{eq:System1} is bounded
  and coercive on $\HOne\times\HOnez$ and hence \eqref{eq:System0}
  and \eqref{eq:System1}
  are well-posed.
\end{remark}
\par
 Again \eqref{eq:System1} implies $(1,\xi_1)_\LT>0$ so that
 \eqref{eq:SystemRepresentation} is
 well-defined.
\par
 Note that \eqref{eq:PFNeumann}, \eqref{eq:cA} and \eqref{eq:System0} imply
\begin{equation*}
  \|\zeta_0\|_\HOne^2+|\xi_0|_\HOne^2\leq
  \cA\big((\zeta_0,\xi_0),(\zeta_0,\xi_0)\big)
  \leq C\|\bof\|_\LT|\zeta_0|_\HOne
\end{equation*}
 and hence we have, by \eqref{eq:PFDirichlet},
\begin{equation}\label{eq:zeta0xi0Bdd}
 \|\zeta_0\|_\HOne+\|\xi_0\|_\HOne\leq C\|\bof\|_\LT.
\end{equation}
\par
 Similarly \eqref{eq:PFNeumann}, \eqref{eq:PFDirichlet},
 \eqref{eq:cA} and \eqref{eq:System1} imply
\begin{equation}\label{eq:zeta1xi1Bdd}
 \|\zeta_1\|_\HOne+\|\xi_1\|_\HOne\leq C.
\end{equation}
\par
 Finally, \eqref{eq:SystemRepresentation},
 \eqref{eq:zeta0xi0Bdd} and \eqref{eq:zeta1xi1Bdd}
 together imply
\begin{equation}\label{eq:xiBddPositive}
\|\xi\|_\HOne\leq C\|\bof\|_\LT.
\end{equation}

\subsection{Problem for $c_1,\ldots,c_m$}\label{subsec:Coefficients}
 It follows from \eqref{eq:QuadCurl}, \eqref{eq:HodgeDecomposition},
 Remark~\ref{rem:Orthogonality} and Remark~\ref{rem:HarmonicFunctions} that
\begin{equation} \label{eq:Coefficients}
  \sum_{j=1}^m(\grad\varphi_i,\grad\varphi_j)_\LT c_j =\gamma^{-1}(\bof,\grad\varphi_i)
  \quad\text{for}\quad 1\leq i\leq m. \end{equation}
\par
  Note that \eqref{eq:Coefficients} is a symmetric system that is  positive definite
  because of the boundary condition \eqref{eq:OuterBoundary}, and it implies
\begin{equation}\label{eq:CoefficientBdd}
  |c_j|\leq C\|\bof\|_\LT \qquad\text{for}\quad 1\leq j\leq m.
\end{equation}
%
\subsection{Summary}\label{subsec:Summary}
 Given $\bof\in\LTT$, we find $\xi$ by solving the problem in
 Section~\ref{subsubsec:gammaZero}
 if $\gamma=0$ or the problem in Section~\ref{subsubsec:gammaPositive} if $\gamma>0$.
 Then we determine
 $\phi$ by solving \eqref{eq:phiEq2}.  In the case where $m\geq 1$,
 we find $\varphi_1,\ldots,\varphi_m$ by solving  \eqref{subeqs:HarmonicFunctions}
 and
 the coefficients $c_1,\ldots,c_m$ by solving \eqref{eq:Coefficients}.
 The solution $\bu$ of \eqref{eq:QuadCurl} is then given by
 \eqref{eq:HodgeDecomposition}.

\section{Virtual Element Methods}\label{sec:VEM}
 We can turn the steps in Section~\ref{subsec:Summary} into
 a numerical procedure by
 using virtual element methods to
 discretize  the boundary value problems, where the two cases
 $\gamma=0$ and $\gamma>1$ will be treated separately.
\subsection{Virtual Element Spaces}\label{subsec:VEM}
 Let $\cT_h$ be a polytopal mesh on $\O$, where the parameter $h$
 stands for the mesh size, which is the maximum of the diameters of
 the polygons in $\cT_h$.
  The diameter of  a polygon $D\in\cT_h$ is
  denoted by $h_D$  and the set of the edges of $D$ is
  denoted by $\ED$.
  We assume the following shape
 regularity on $D\in\cT_h$:
 There exists a constant $\Theta \in (0,1)$ independent of
 $h$ such that
\begin{subequations}\label{subeqs:MA}
    \begin{align}
        &\mbox{$D$ is star-shaped with respect to a ball of radius
        $\Theta h_D$, and}  \label{eq:MA1}\\
        &\mbox{$|e| \geq \Theta h_D$ for any edge $e \in \ED$.}
        \label{eq:MA2}
        \end{align}
\end{subequations}
\par
 The local enhanced virtual element space $V_h^k(D) \subset H^1(D)$
  (cf. \cite{AABMR:2013:Projector}) is defined as follows:
\begin{align*}
    V_h^k(D) &= \{ v_h \in H^1(D): v_h|_{\p D} \in \Poly_{k}(\p D), \;\;
    \Delta v_h \in \Poly_{k}(D), \\
     &\hspace{60pt} \Pi^0_{k,D} v_h - \Pi^{1}_{k,D} v_h \in
     \Poly_{k-2}(D)\}.     \notag
\end{align*}
 Here $\Poly_k(\p D)$ is the space of continuous piecewise polynomials
 of degree at most $k$ on $\p D$
 and $\Poly_k(D)$ is the space of polynomials of degree at most $k$ on $D$.
   The operator $\Pi_{k,D}^0$
  is the orthogonal projection from $\LT$ onto $\Poly_k(D)$, and the operator
   $\Pi_{k,D}^1$ is the orthogonal
  projection from $H^1(D)$ onto $\Poly_k(D)$ with respect to the bilinear form
\begin{align*}
  (\!(\zeta,\eta)\!)&=(\bcurl\zeta,\bcurl\eta)_\LTD+
     \Big(\int_{\p D}\zeta\,ds\Big)\Big(\int_{\p D}\eta\,ds\Big)\\
     &=(\grad\zeta,\grad\eta)_\LTD+\Big(\int_{\p D}\zeta\,ds\Big)
     \Big(\int_{\p D}\eta\,ds\Big).
\end{align*}
\par
 We will restrict $k$ to be $1$ or $2$ (cf. Remark~\ref{rem:Limit}).
   For $k=1$, the degrees
 of freedom (dofs) of $v_h\in V_h(D)$
 consist of the values of $v_h$ at the vertices of $D$.  For $k=2$,
 the dofs of $v_h\in V_h(D)$
 consist of the values of $v_h$ at the vertices of $D$, the values
 of $v_h$ at the midpoints of the edges in $\ED$, and
 $(v_h,1)_\LTD$.
\begin{remark} \label{rem:LTProjection}
  Both $\Pi_{k,D}^1 v_h$ and $\Pi_{k,D}^0 v_h$ are computable in terms
  of the degrees of freedom
  of $v_h\in V_h(D)$ (cf. \cite{AABMR:2013:Projector,BGS:2017:VEM2}).
\end{remark}
\par
 The global virtual element spaces $V_h^k$ and $\zV_h^k$ are defined by
 \begin{align*}
     V_h^k = \{ v_h \in H^1(\Omega) : v_h|_D \in V_h(D)
     \text{ for all } D \in \Th \}, \\
    \zV_h^k = \{ v_h \in H_0^1(\Omega) : v_h|_D \in V_h(D)
    \text{ for all } D \in \Th \}.
 \end{align*}
\par
  Let $\Poly_k(\O;\cT_h)$ be the space of piecewise polynomials of
  degree at most $k$ with respect to $\cT_h$.
 The global version of $\Pi_{k,D}^0$ is the operator
 $\PiZ:\LT\rightarrow\Poly_k(\O;\cT_h)$ defined by
\begin{alignat*}{3}
 (\PiZ \zeta)\big|_D&=\Pi_{k,D}^0(\zeta\big|_D) &\qquad&\forall\,\zeta\in\LT,
  \intertext{and the global version of $\Pi_{k,D}^1$ is the operator
  $\PiO:\HOne\longrightarrow\Poly_k(\O;\cT_h)$ defined by}
 (\PiO \zeta)\big|_D&=\Pi_{k,D}^1(\zeta\big|_D) &\qquad&\forall\,\zeta\in\HOne.
\end{alignat*}
\begin{remark}\label{rem:OtherNotation}
  Another notation for $\Pi_{k,D}^1$  (resp., $\PiO$) is
  $\Pi_{k,D}^{\nabla}$ (resp., $\Pi_{k,h}^{\nabla}$).
\end{remark}
%
\subsection{Virtual Element Methods for \eqref{eq:QuadCurl}
 in the Case  $\gamma=0$} \label{subsec:VEMzero}
 According to the assumption on $\gamma$, this means $\O$ is simply
 connected, i.e., $m=0$,
  and the Hodge decomposition of $\bu$ is simplified to
\begin{equation}\label{eq:SimplifiedHodge}
  \bu=\bcurl\phi.
\end{equation}
%
\par
 First we discretize the boundary value problems in
 Section~\ref{subsubsec:gammaZero} as follows.
\begin{itemize}\itemsep 4pt
  \item  Find $\rho_h\in V_h^k$ such that
 \begin{equation}\label{eq:rhoh}
  a_h(\rho_h, \psi_h) + (\rho_h,1)_\LT(\psi_h,1)_\LT
   = (\bof, \bcurl \Pi^{1}_{k,h} \psi_h)_\LT \quad \forall \psi_h \in V_h^k,
 \end{equation}
 where the bilinear form $a_h(\cdot,\cdot)$ is defined by
\begin{align}\label{eq:ah}
       &\,a_h(w_h, v_h) \notag\\
     =&\, \sum_{D \in \Th}
     (\bcurl \Pi^{1}_{k,D} w_h, \bcurl \Pi^{1}_{k,D} v_h)_\LTD
      + S^D\big((\Id - \Pi^{1}_{k,D}) w_h, (\Id - \Pi^{1}_{k,D}) v_h\big)\\
      =&\, \sum_{D \in \Th}
       (\grad \Pi^{1}_{k,D} w_h, \grad \Pi^{1}_{k,D} v_h)_\LTD
      + S^D\big((\Id - \Pi^{1}_{k,D}) w_h, (\Id - \Pi^{1}_{k,D}) v_h\big).\notag
\end{align}
 Here $\Id$ is the identity operator and the symmetric positive definite
  bilinear form $S^D(\cdot,\cdot)$ is given by
\begin{equation*}
  S^D(w,v)=\sum_{p\in \cN_{\p D}} w(p)v(p),
\end{equation*}
 where $\cN_{\p D}$ is the set of nodes on $\p D$ associated with the
 dofs of $V_h^k(D)$, i.e.,
 the set of vertices of $D$  if $k=1$ and the set of vertices of $D$
 together with the
 midpoints of the edges in $\ED$ if $k=2$.
 \par\noindent
 Note that \eqref{eq:rhoh} implies
\begin{equation}\label{eq:rhohZeroMean}
 (\rho_h,1)_\LT=0.
\end{equation}
\item Find $\xi_h\in \zV_h^k$ given by
\begin{alignat}{3}
&\xi_h=\xi_{0,h}-[(1,\xi_{0,h})_\LT/(1,\xi_{1,h})_\LT]\xi_{1,h},
\label{eq:xih}
  \intertext{where $\xi_{0,h},\xi_{1,h}\in \zV_h^k$ are determined by}
  &a_h(\xi_{0,h},\eta_h)+\beta(\PiZ\xi_{0,h},\PiZ\eta_h)_\LT
  =(\PiZ\rho_h,\PiZ\eta_h)_\LT&\qquad&\forall\,\eta_h\in \zV_h^k,\label{eq:xi0h}\\
    &a_h(\xi_{1,h},\eta_h)+\beta(\PiZ\xi_{1,h},\PiZ\eta_h)_\LT
    =(1,\PiZ\eta_h)_\LT&\qquad&\forall\,\eta_h\in \zV_h^k.\label{eq:xi1h}
\end{alignat}
\end{itemize}
\par
 Next we find $\phi_h\in V_h^k$ such that
\begin{equation}\label{eq:phih}
 a_h(\phi_h,\psi_h)+(\phi_h,1)_\LT(\psi_h,1)_\LT
  =(\PiZ\xi_h,\PiZ\psi_h)_\LT   \qquad\forall\,\psi_h\in V_h^k.
\end{equation}
\par
 In view of \eqref{eq:SimplifiedHodge}, the approximation of $\bu$ is
 then given by
\begin{equation}\label{eq:uh1}
\bu_h=\bcurl(\PiO\phi_h).
\end{equation}

\subsection{Virtual Element Methods for \eqref{eq:QuadCurl}
 in the Case  $\gamma>0$}
 \label{subsec:VEMpositive}
  First we discretize the boundary value problem in
  Section~\ref{subsubsec:gammaPositive} as follows.
\par
 Find $(\zeta_h,\xi_h)\in V_h^k\times \zV_h^k$ given by
\begin{alignat}{3}
&(\zeta_h,\xi_h)=(\zeta_{0,h},\xi_{0,h})-
  [(1,\xi_{0,h})_\LT/(1,\xi_{1,h})_\LT](\zeta_{1,h},\xi_{1,h}),\label{eq:zetahxih}
\intertext{where $(\zeta_{0,h},\xi_{0,h}), (\zeta_{1,h},\xi_{1,h})\in
V_h^k\times \zV_h^k$ are determined by}
&\cA_h\big((\zeta_{0,h},\xi_{0,h}),(\psi_h,\eta_h)\big)
  +(\zeta_{0,h},1)_\LT ( \psi_{h},1)_\LT=
  \gamma^{-\frac12}(\bof,\bcurl\PiO\psi_h)_\LT,
  \label{eq:zetahxih0}     \\
&\cA_h\big((\zeta_{1,h},\xi_{1,h}),(\psi_h,\eta_h)\big)
  +(\zeta_{1,h},1)_\LT ( \psi_{h},1)_\LT=(1,\PiZ\eta_h)_\LT,
  \label{eq:zetahxih1}
\end{alignat}
 for all $(\psi_h,\eta_h)\in V_h^k\times \zV_h^k$,
 and the bilinear form $\cA_h(\cdot,\cdot)$ is given by
\begin{align}\label{eq:cAh}
 &\cA_h\big((\sigma_h,\mu_h),(\psi_h,\eta_h)\big)
  =a_h(\sigma_h,\psi_h)+\gamma^\frac12(\PiZ\psi_h,\PiZ\mu_h)_\LT \\
  &\hspace{50pt}-\gamma^\frac12(\PiZ\sigma_h,\PiZ\eta_h)_\LT
  +a_h(\mu_h,\eta_h)
 +\beta(\PiZ\mu_h,\PiZ\eta_h)_\LT.\notag
\end{align}
\par
 Next we determine $\phi_h$ by \eqref{eq:phih}.
 If $\O$ is simply connected (i.e., $m=0$), then the approximation
 of $\bu$ is given by  \eqref{eq:uh1}.
\par
 If $\O$ is not simply connected (i.e., $m\geq1$), we compute
 $\varphi_{j,h}\in V_h^k$ $(1\leq j\leq m$)
 such that
\begin{subequations}\label{subeqs:DHarmonicFunctions}
\begin{alignat}{3}
  a_h(\varphi_{j,h},v_h)&=0&\qquad&\forall\,v_h\in \zV_h^k,
  \label{eq:DHarmonic}   \\
  \varphi_{j,h}\Big|_{\G_0}&=0,\label{eq:DOuterBoundary}\\
  \varphi_{j,h}\Big|_{\G_k}&=\begin{cases}
    1 &\qquad\text{if $k=j$}\\[4pt]
    0&\qquad\text{if $k\neq j$}
  \end{cases}
  &\qquad&\text{for}\;1\leq j,k\leq m,
  \label{eq:DInnerBoundary}
\end{alignat}
\end{subequations}
 and we find $c_{1,h}, \ldots,c_{m,h}\in\R$ by
 solving the SPD system
\begin{equation}\label{eq:DCoefficients}
 \sum_{j=1}^m a_h(\varphi_{i,h},\varphi_{j,h})c_{j,h}
  =\gamma^{-1}(\bof,\grad \PiO\varphi_{i,h})_\LT
  \qquad\text{for}\quad 1\leq i\leq m.
\end{equation}
\par
 The approximation of $\bu$ is then given by
\begin{equation}\label{eq:uh2}
  \bu_h=\bcurl(\PiO\phi_h)+\sum_{j=1}^m c_{j,h}\,\grad(\PiO\varphi_{j,h}).
\end{equation}
\subsection{Properties of Virtual Elements}\label{subsec:VEMProperties}
 We recall some known properties of the virtual element methods in this section.
 The derivations of these properties under the shape regularity assumption
 \eqref{subeqs:MA} can be found for example in \cite{BGS:2017:VEM2,BSung:2018:VEM}.
\begin{remark}\label{rem:SobolevSpaces}
  The results in \cite{BGS:2017:VEM2,BSung:2018:VEM} are based on the
  Bramble-Hilbert lemma in \cite{BH:1970:Lemma} for Sobolev spaces of integer orders.
  They can be extended to fractional order Sobolev spaces in a straightforward way
   by using the version of the Bramble-Hilbert lemma in \cite{DS:1980:BH}.
\end{remark}

\subsubsection{Properties of $\Pi_{k,D}^1$}\label{subsubsec:PiO}
 We have (by the definition of $\PiO$)
\begin{alignat}{3}\label{eq:PiOProjection}
 (\grad \zeta,\grad q)_\LTD&=(\bcurl\zeta,\bcurl q)_\LTD\notag\\
             &=(\bcurl\Pi_{k,D}^1\zeta,\bcurl q)_\LTD\\
             &=(\grad\Pi_{k,D}^1\zeta,\grad q)_\LT&\qquad&\forall\,
             \zeta\in H^1(D),\;q\in \Poly_k(D),\notag
\end{alignat}
\begin{alignat}{3}\label{eq:PiOPythagoras}
   \|\grad\zeta\|_\LTD^2&=\|\grad \Pi_{k,D}^1 \zeta\|_\LTD^2+
      \|\grad(\Id-\Pi_{k,D}^1)\zeta)\|_\LTD^2\notag\\
 &=\|\bcurl \Pi_{k,D}^1\zeta\|_\LTD^2+\|\bcurl(\Id-\Pi_{k,D}^1)\zeta)\|_\LTD^2\\
 &=\|\bcurl\zeta\|_\LTD^2 &\qquad&\forall\,\zeta \in H^1(D), \notag
\end{alignat}
 and (cf. \cite[Section~3.5]{BSung:2018:VEM})
\begin{equation}\label{eq:PiOErrorEstimate}
  |\zeta-\Pi_{k,D}^1\zeta|_{H^{1+t}(D)}\leq C_{s,t} h_D^{s-t} |\zeta|_{H^{1+s}(D)}
  \qquad\forall\,\zeta\in H^{1+s}(D),\;0\leq t<s\leq k.
\end{equation}

\subsubsection{Properties of  $S^D(\cdot,\cdot)$}\label{subsubsec:SD}
 We have, by \cite[Lemma~3.11 and Lemma~3.18]{BSung:2018:VEM}
  and also \cite[Lemma~2.11 and Lemma~2.22]{BGS:2017:VEM2},
\begin{equation}\label{eq:SDBound}
 C_\flat|v_h-\Pi_{k,D}^1 v_h|_\HOneD^2\leq  S^D\big((\Id-\Pi_{k,D}^1)
 v_h,(\Id-\Pi_{k,D}^1) v_h\big)
 \leq C_\sharp |v_h-\Pi_{k,D}^1 v_h|_\HOneD^2
\end{equation}
 for all $v_h\in V_h^k(D)$.

\subsubsection{Interpolation Error Estimates}\label{subsubsec:Interpolation}
 Let $\lambda$ belong to $H^{1+s}(D)$ for some $s>0$. The interpolant
 $I_{k,D}\lambda\in V_h^k(D)$
 is defined by the condition that $\lambda$ and $I_{k,D}\lambda$ share
 the same dofs of $V_h^k(D)$.
    We have (cf. \cite[Section~3.8]{BSung:2018:VEM})
\begin{align}
 &|\lambda-I_{k,D}\lambda|_\HOneD+|\lambda-\Pi_{k,D}^1 I_{k,D}\lambda|_\HOneD
    \leq C_s h^{\min(s,k)}|\lambda|_{H^{1+s}(D)},
 \label{eq:InterpolationOne} \\
 &\|\lambda-I_{k,D}\lambda\|_\LTD+\|\lambda- \Pi_{k,D}^1 I_{k,D}\lambda\|_\LTD
 \label{eq:InterpolationZero}\\
     &\hspace{70pt}+\|\lambda-\Pi_{k,D}^0 I_{k,D}\lambda\|_\LTD
     \leq C_s h^{1+\min(s,k)}|\lambda|_{H^{1+s}(D)}.
   \notag
 \end{align}
\subsubsection{A Coercivity Estimate}\label{subsubsec:Coercivity}
 It follows from \eqref{eq:ah}, \eqref{eq:PiOPythagoras}
 and \eqref{eq:SDBound}
 that
\begin{equation} \label{eq:Coercivity}
 \|\grad v_h\|_\LT^2=\|\bcurl v_h\|_\LT^2\leq \Cc\, a_h(v_h,v_h)
 \qquad\forall\,v_h\in V_h^k.
\end{equation}
%
\subsubsection{An Inconsistency Estimate}\label{subsubsec:Inconsistency}
 Given any $v_h,w_h\in V_h^k$, we have an identity
\begin{align}\label{eq:Identity}
  &\hspace{6pt} a_h(v_h,w_h)-(\bcurl v_h,\bcurl w_h)_\LT\\
  =&\,\sum_{D\in\cT_h}\big[-\big(\bcurl v_h,\bcurl(\Id-\Pi_{k,D}^1)w_h\big)_\LTD
  +S^D\big((\Id-\Pi_{k,D}^1)v_h,(\Id-\Pi_{k,D}^1)w_h\big)\big]
  \notag
\end{align}
 that follows from \eqref{eq:ah} and \eqref{eq:PiOProjection}, which together with
  \eqref{eq:PiOPythagoras} and \eqref{eq:SDBound}
 implies
\begin{equation}\label{eq:Inconsistency}
  \big| a_h(v_h,w_h)-(\bcurl v_h,\bcurl w_h)_\LT|\leq
  C_{\rm inc}| v_h|_\HOne|(\Id-\PiO) w_h|_{h,1},
\end{equation}
where $|\cdot|_{h,1}$ is the piecewise $H^1$ norm defined by
\begin{equation*}
 |v|_{h,1}^2=\sum_{D\in\cT_h}|v|_{H^1(D)}^2 \;(=|v|_\HOne^2 \;\text{for $v\in\HOne$}).
\end{equation*}
\begin{remark}\label{rem:Grad}
  The identity \eqref{eq:Identity} and the estimate \eqref{eq:Inconsistency}
  remain valid if $\bcurl$ is replaced by $\grad$.
\end{remark}
\section{Error Analysis of $\xi_h$}\label{sec:xihError}
 The error estimate for $\bu_h$ involves error estimates for $\xi_h$ and $\phi_h$, and
 also error estimates  for $\varphi_{j,h}$ and $c_{j,h}$ if $m\geq 1$.
 We will carry out the error analysis of $\xi_h$ in this section.
 Below the interpolation operator $I_{k,h}:H^{1+s}(\O)\longrightarrow V_h^k$ ($s>0$)
  is the global version of $I_{k,D}$, i.e.,
\begin{equation*}
  (I_{k,h}\lambda)\big|_D=I_{k,D}(\lambda\big|_D) \qquad \forall\,D\in\cT_h.
\end{equation*}
\par
 Our goal is to establish the following result where $\xi=\curl\bu$.
\begin{theorem}\label{thm:xihError}
 Let $\omega$ be the largest interior angle of $\O$.  We have
\begin{subequations}\label{eq:xihError}
\begin{align}
 |\xi-\xi_h|_\HOne+|\xi-\PiO\xi_h|_{h,1}&\leq \CO \hfactorG\|\bof\|_\LT
  \label{eq:xihErrorG}
  \intertext{if $(\pi/\omega)$ is not an integer, and}
   |\xi-\xi_h|_\HOne+|\xi-\PiO\xi_h|_{h,1}&\leq \COE \hfactorE\|\bof\|_\LT
   \label{eq:xihErrorE}
\end{align}
 if $(\pi/\omega)$ is an integer.
\end{subequations}
\end{theorem}
\begin{remark}\label{rem:Consequence}
 Note that \eqref{eq:PFDirichlet}, \eqref{eq:xiBddZero}
 (for the case where $\gamma=0$),
  \eqref{eq:xiBddPositive} (for the case where $\gamma>0$)
  and \eqref{eq:xihError}  imply
\begin{equation} \label{eq:xihBdd}
  \|\xi_h\|_\HOne\leq C\|\bof\|_\LT.
\end{equation}
\end{remark}
\par
 We will prove Theorem~\ref{thm:xihError} for the two cases
 $\gamma=0$ and $\gamma>0$ separately, where
  we adopt the notation of $\ExpoStar$ and $\CExpo$ in
 Remark~\ref{rem:RegularityNotation}.
%
\subsection{The Case Where $\gamma=0$}\label{subsec:xihErrorZero}
 First we establish a duality estimate for $\rho_h\in V_h^k$ defined
 by \eqref{eq:rhoh}.
%
\subsubsection{A Duality Estimate} \label{subsub:DualityZero}
\begin{lemma}\label{lem:Duality}
 We have
\begin{equation}\label{eq:Duality}
 |(\rho-\rho_h,\chi)_\LT|\leq \CExpo \hfactor \|\bof\|_\LT \|\chi\|_\HOne
  \qquad\forall\,\chi\in \HOne.
\end{equation}
\end{lemma}
\begin{proof}
 From  \eqref{eq:rhoh}, \eqref{eq:rhohZeroMean}, \eqref{eq:PiOPythagoras} and
 \eqref{eq:Coercivity} we have
\begin{align*}
  \|\bcurl\rho_h\|_\LT^2&\leq \Cc a_h(\rho_h,\rho_h) \\
 & = \Cc(\bof,\bcurl\PiO\rho_h)_\LT\\
  &\leq \Cc\|\bof\|_\LT\|\bcurl \PiO\rho_h\|_\LT
  \leq \Cc\|\bof\|_\LT\|\bcurl\rho_h\|_\LT, \end{align*}
 which implies the following analog of \eqref{eq:rhoBdd}:
\begin{equation}\label{eq:Duality2}
  \|\bcurl\rho_h\|_\LT\leq \Cc\|\bof\|_\LT.
\end{equation}
\par
 Let $\chi\in\HOne$ be arbitrary and $\lambda\in\HOne$ be defined by
 the Neumann boundary value problem
\begin{equation}\label{eq:Dulaity3}
 (\bcurl\psi,\bcurl\lambda)_\LT+(\psi,1)_\LT(\lambda,1)_\LT
 =(\psi,\chi)_\LT
 \qquad \forall\,\psi\in\HOne.
\end{equation}
 Then we have
\begin{equation}\label{eq:Duality4}
  \|\lambda\|_\RegStar\leq \CExpo \|\chi\|_\HOne
\end{equation}
  by elliptic regularity (cf. \cite[Section~5.1]{Grisvard:1985:EPN} and
 \cite[Section~2.5]{Dauge:1988:EBV}).
\par
 Since $\rho,\rho_h\in \LTz$, we find from \eqref{eq:Dulaity3} that
\begin{align} \label{eq:Duality5}
  &(\rho-\rho_h,\chi)_\LT=\big(\bcurl\lambda,\bcurl(\rho-\rho_h)\big)_\LT\\
&\hspace{40pt}
=\big(\bcurl(\lambda-I_{k,h}\lambda),\bcurl(\rho-\rho_h)\big)_\LT
       +\big(\bcurl I_{k,h}\lambda,\bcurl(\rho-\rho_h)\big)_\LT,\notag
\end{align}
 and we have
\begin{equation}\label{eq:Dulaity6}
  \big|\big(\bcurl(\lambda-I_{k,h}\lambda),\bcurl(\rho-\rho_h)\big)_\LT\big|
  \leq \CExpo  \hfactor \|\bof\|_\LT \|\chi\|_\HOne
\end{equation}
 by \eqref{eq:rhoBdd}, \eqref{eq:InterpolationOne},  \eqref{eq:Duality2}
 and \eqref{eq:Duality4}.
\par
 Next we analyze the second term on the right-hand side of \eqref{eq:Duality5}.
  We obtain from \eqref{eq:rhoEquation}  and \eqref{eq:rhoh}
  that
\begin{align}\label{eq:Duality7}
   &\big(\bcurl I_{k,h}\lambda,\bcurl(\rho-\rho_h)\big)_\LT
 =(\bof,\bcurl I_{k,h}\lambda)_\LT-(\bcurl\rho_h,\bcurl I_{k,h}\lambda)_\LT\\
  &\hspace{30pt}=\big(\bof,\bcurl(I_{k,h}\lambda-\PiO I_{k,h}\lambda)\big)_\LT
   +a_h(\rho_h,I_{k,h}\lambda)-(\bcurl\rho_h,\bcurl I_{k,h}\lambda)_\LT.\notag
\end{align}
\par
 We can use \eqref{eq:InterpolationOne}, \eqref{eq:Inconsistency},
 \eqref{eq:Duality2} and
  \eqref{eq:Duality4} to estimate the terms on the right-hand side of
  \eqref{eq:Duality7}  as follows:
\begin{align}
 &\big|\big(\bof,\curl(I_{k,h}\lambda-\PiO I_{k,h}\lambda)\big)_\LT|\leq \CExpo \hfactor
 \|\bof\|_\LT\|\chi\|_\HOne,\label{eq:Duality8}\\
 \intertext{and}
 &\big|a_h(\rho_h,I_{k,h}\lambda)-(\curl\rho_h,\curl I_{k,h}\lambda)_\LT\big|
 \label{eq:Duality9}\\
 &\hspace{50pt}
 \leq C_{\rm inc}|\rho_h|_\HOne|(\Id-\PiO)I_{k,h}\lambda|_{h,1}
 \leq \CExpo \hfactor\|\bof\|_\LT \|\chi\|_\HOne.
 \notag
\end{align}
\par
 The estimate \eqref{eq:Duality} follows from \eqref{eq:Duality5}--\eqref{eq:Duality9}.
\end{proof}
%
\subsubsection{Estimate for $\xi_h$}\label{subsubsec:xihEstimateZero}
 We first consider $\xi_{0,h}$.
\par
  Let $\tilde\xi_{0}\in\HOnez$ be defined by the Dirichlet boundary value problem
\begin{equation}\label{eq:ZeroError1}
  (\bcurl\tilde\xi_0,\bcurl\eta)_\LT+\beta(\tilde\xi_0,\eta)_\LT=(\rho_h,\eta)_\LT
   \qquad\forall\,\eta\in\HOnez.
\end{equation}
 Then we have
\begin{equation}\label{eq:ZeroError2}
  \|\tilde\xi_0\|_\RegStar\leq \CExpo
  \|\rho_h\|_\HOne\leq \CExpo\|\bof\|_\LT
\end{equation}
 by elliptic regularity and \eqref{eq:Duality2}.
\par
 Comparing \eqref{eq:xi0} and \eqref{eq:ZeroError1}, we see that
\begin{equation*}
  \big(\bcurl(\xi_0-\tilde\xi_0),\bcurl\eta)_\LT
  +\beta(\xi_0-\tilde\xi_0,\eta)_\LT=(\rho-\rho_h,\eta)_\LT \qquad\forall\,
  \eta \in \HOnez,
\end{equation*}
 which together with \eqref{eq:PFDirichlet} and Lemma~\ref{lem:Duality}
  implies
\begin{align*}
  |\xi_0-\tilde\xi_0|_\HOne^2&=
  \big(\bcurl(\xi_0-\tilde\xi_0),\bcurl(\xi_0-\tilde\xi_0)\big)_\LT\notag\\
    &\leq (\rho-\rho_h,\xi_0-\tilde\xi_0)_\LT  \\
    &\leq \CExpo \hfactor \|\bof\|_\LT |\xi_0-\tilde\xi_0|_\HOne,
\end{align*}
 and hence
\begin{equation}\label{eq:ZeroError3}
  |\xi_0-\tilde\xi_0|_\HOne\leq \CExpo\hfactor \|\bof\|_\LT.
\end{equation}
\par
 On the other hand, $\xi_{0,h}$ defined by \eqref{eq:xi0h} is an approximation of
 $\tilde\xi_0$ generated by a virtual element method for the Dirichlet
 boundary value problem \eqref{eq:ZeroError1}
 whose right-hand side belongs to $\HOne$.  Therefore we have,
 in view of \eqref{eq:ZeroError2},
\begin{equation}\label{eq:ZeroError4}
 |\tilde\xi_0-\xi_{0,h}|_\HOne+  |\tilde\xi_0-\PiO\xi_{0,h}|_{h,1}
 \leq \CExpo\hfactor \|\bof\|_\LT
\end{equation}
 (cf. Appendix~\ref{append:VEMError}).
\par
 It follows from \eqref{eq:ZeroError3} and \eqref{eq:ZeroError4} that
\begin{equation}\label{eq:ZeroError5}
 |\xi_0-\xi_{0,h}|_\HOne+ |\xi_0-\PiO \xi_{0,h}|_{h,1}\leq
 \CExpo\hfactor \|\bof\|_\LT.
\end{equation}
\par
 Next we consider $\xi_{1,h}$.  By comparing \eqref{eq:xi1}
   and \eqref{eq:xi1h}, we see that $\xi_{1,h}$ is
 the approximation of $\xi_1$ obtained by a virtual element method for
 the Dirichlet boundary value problem \eqref{eq:xi1} whose right-hand side belongs
 to $\HOne$.
 Moreover \eqref{eq:xi1} and elliptic regularity imply
\begin{equation*}
  \|\xi_1\|_\RegStar\leq \CExpo,
\end{equation*}
 and hence
\begin{equation}\label{eq:ZeroError6}
  |\xi_1-\xi_{1,h}|_\HOne+|\xi_1-\PiO\xi_{1,h}|_{h,1}\leq \CExpo \hfactor
\end{equation}
 (cf. Appendix~\ref{append:VEMError}).
\par
 The estimate \eqref{eq:xihError}  follows from \eqref{eq:xiRepresentation},
 \eqref{eq:xih}, \eqref{eq:ZeroError5}, \eqref{eq:ZeroError6}     and the fact that
   $$|(1,\xi_0)_\LT|\leq C\|\bof\|_\LT$$
 because of \eqref{eq:xi0Bdd}.
\subsection{The Case Where $\gamma>0$}\label{subsec:xihErrorPositive}
 We begin with duality estimates for $\zeta_{0,h}, \zeta_{1,h}\in V_h^k$ defined by
 \eqref{eq:zetahxih0} and \eqref{eq:zetahxih1}.
%
\subsubsection{Duality Estimates}\label{subsubsec:DualityPositive}
\begin{lemma}\label{lem:SystemDuality}
  We have
\begin{alignat}{3}
  |(\zeta_0-\zeta_{0,h},\chi)_\LT|&\leq \CExpo \hfactor \|\bof\|_\LT\|\chi\|_\HOne
  &\qquad&  \forall\,\chi\in\HOne, \label{eq:SystemDualityZero} \\
  |(\zeta_1-\zeta_{1,h},\chi)_\LT|&\leq \CExpo \hfactor\|\chi\|_\HOne&\qquad&\forall\,
  \chi\in\HOne.\label{eq:SystemDualityOne}
\end{alignat}
\end{lemma}
\begin{proof} We will focus on \eqref{eq:SystemDualityZero} since the derivation of
\eqref{eq:SystemDualityOne} is very similar.
\par
   From \eqref{eq:zetahxih0}, \eqref{eq:cAh}, \eqref{eq:PiOPythagoras}
    and \eqref{eq:Coercivity},     we have
\begin{align*}
  &(\zeta_{0,h},1)_\LT^2+\|\bcurl\zeta_{0,h}\|_\LT^2+\|\bcurl\xi_{0,h}\|_\LT^2\\
  &\hspace{50pt}\leq
    C   \big[a_h(\zeta_{0,h},\zeta_{0,h})+a_h(\xi_{0,h},\xi_{0,h})
    +(\zeta_{0,h},1)_\LT^2\big] \\
&\hspace{50pt}\leq C\big[\cA\big((\zeta_{0,h},\xi_{0,h}),(\zeta_{0,h},\xi_{0,h})\big)
  +(\zeta_{0,h},1)_\LT(\zeta_{0,h},1)_\LT\big]\\
  &\hspace{50pt}
  \leq C\|\bof\|_\LT\|\bcurl\zeta_{0,h}\|_\LT,
\end{align*}
 which together with \eqref{eq:PFNeumann} and
 \eqref{eq:PFDirichlet} implies the following analog of \eqref{eq:zeta0xi0Bdd}:
\begin{equation}\label{eq:SystemDuality1}
  \|\zeta_{0,h}\|_\HOne+\|\xi_{0,h}\|_\HOne\leq C\|\bof\|_\LT.
\end{equation}
\par
 Let $\chi\in\HOne$ be arbitrary and
 $(\lambda,\mu)\in \HOne\times \HOnez$ be defined by
\begin{equation}\label{eq:SystemDuality2}
  \cA\big((\psi,\eta),(\lambda,\mu)\big)+(\psi,1)_\LT(\lambda,1)_\LT
  =(\psi,\chi)_\LT
  \qquad\forall\,(\psi,\eta)\in\HOne\times\HOnez,
\end{equation}
 or equivalently, in view of \eqref{eq:cA},
\begin{align}
(\bcurl\psi,\bcurl\lambda)_\LT +(\psi,1)_\LT(\lambda,1)_\LT
 &=\gamma^\frac12(\psi,\mu)_\LT+(\psi,\chi)_\LT\label{eq:SystemDuality3}
 \intertext{for all  $\psi\in\HOne$,}
 (\bcurl\eta,\bcurl\mu)_\LT+\beta(\eta,\mu)_\LT&=-\gamma^{-\frac12}(\lambda,\eta)_\LT
 \label{eq:SystemDuality4}
 \end{align}
 for all $\eta\in\HOnez$.
\par
 Note that \eqref{eq:PFNeumann}, \eqref{eq:PFDirichlet},
 \eqref{eq:cA} and \eqref{eq:SystemDuality2} imply
\begin{equation*}
  \|\lambda\|_\HOne+\|\mu\|_\HOne\leq C\|\chi\|_\LT, \end{equation*}
 and hence we have
\begin{equation}\label{eq:SystemDuality5}
  \|\lambda\|_\RegStar+
   \|\mu\|_\RegStar \leq \CExpo \|\chi\|_\HOne
\end{equation}
 by \eqref{eq:SystemDuality3}, \eqref{eq:SystemDuality4}
 and  elliptic regularity.
\par
 The function $\zeta_0-\zeta_{0,h}\in\HOne$ satisfies, because of
 \eqref{eq:SystemDuality2},
\begin{align}\label{eq:SystemDuality6} &\,(\zeta_0-\zeta_{0,h},\chi)_\LT\notag\\
  =&\,\cA\big((\zeta_0-\zeta_{0,h},\xi_0-\xi_{0,h}),
  (\lambda,\mu)\big)+(\zeta_0-\zeta_{0,h},1)_\LT(\lambda,1)_\LT\notag\\
=&\,\cA\big((\zeta_0-\zeta_{0,h},\xi_0-\xi_{0,h}),(\lambda-I_{k,h}\lambda,\mu-I_{k,h}\mu)\big)
     \notag\\
     &\hspace{20pt}+(\zeta_0-\zeta_{0,h},1)_\LT(\lambda-I_{k,h}\lambda,1)_\LT   \\
    &\hspace{40pt}+\cA\big((\zeta_0-\zeta_{0,h},\xi_0-\xi_{0,h}),(I_{k,h}\lambda,I_{k,h}\mu)\big)
    +(\zeta_0-\zeta_{0,h},1)_\LT(I_{k,h}\lambda,1)_\LT, \notag
\end{align}
 and we have
\begin{align}\label{eq:SystemDuality7}
  &\big|\cA\big((\zeta_0-\zeta_{0,h},\xi_0-\xi_{0,h}),(\lambda-I_{k,h}\lambda,\mu-I_{k,h}\mu)\big)
     +(\zeta_0-\zeta_{0,h},1)_\LT(\lambda-I_{k,h}\lambda,1)_\LT\big|\\
     &\hspace{40pt}\leq \CExpo \hfactor\|\bof\|_\LT \|\chi\|_\LT\notag
\end{align}
 by \eqref{eq:cA}, \eqref{eq:zeta0xi0Bdd},
  \eqref{eq:InterpolationOne}, \eqref{eq:InterpolationZero},
  \eqref{eq:SystemDuality1} and \eqref{eq:SystemDuality5}.
\par
 Next we use \eqref{eq:System0} and \eqref{eq:zetahxih0} to  write
\begin{align}\label{eq:SystemDuality8}
&\,\cA\big((\zeta_0-\zeta_{0,h},\xi_0-\xi_{0,h}),(I_{k,h}\lambda,I_{k,h}\mu)\big)
     +(\zeta_0-\zeta_{0,h},1)_\LT(I_{k,h}\lambda,1)_\LT\notag\\
    =&\,\gamma^{-\frac12}(\bof,\bcurl I_{k,h}\lambda)_\LT
    -\cA\big((\zeta_{0,h},\xi_{0,h}),(I_{k,h}\lambda,I_{k,h}\mu)\big)
    -(\zeta_{0,h},1)_\LT(I_{k,h}\lambda,1)_\LT\notag \\
 =&\,\gamma^{-\frac12}\big(\bof,\bcurl (I_{k,h}\lambda-\PiO I_{k,h}\lambda)\big)_\LT+
    \cA_h\big((\zeta_{0,h},\xi_{0,h}),(I_{k,h}\lambda,I_{k,h}\mu)\big)\\
    &\hspace{40pt}
    -\cA\big((\zeta_{0,h},\xi_{0,h}),(I_{k,h}\lambda,I_{k,h}\mu)\big),  \notag
\end{align}
 and observe that
\begin{equation}\label{eq:SystemDuality9}
  \big|\gamma^{-\frac12}\big(\bof,\curl (I_{k,h}\lambda-\PiO I_{k,h}\lambda)\big)_\LT\big|
  \leq \CExpo \hfactor \|\bof\|_\LT\|\chi\|_\HOne
\end{equation}
 by \eqref{eq:InterpolationOne} and \eqref{eq:SystemDuality5}.
\par
 Using  \eqref{eq:cA} and \eqref{eq:cAh}, we can write
\begin{align}\label{eq:SystemDulaity10}
  &\,\cA_h\big((\zeta_{0,h},\xi_{0,h}),(I_{k,h}\lambda,I_{k,h}\mu)\big)
    -\cA\big((\zeta_{0,h},\xi_{0,h}),(I_{k,h}\lambda,I_{k,h}\mu)\big)\notag\\
    =&\,\big[a_h(\zeta_{0,h},I_{k,h}\lambda)-(\bcurl\zeta_{0,h},\bcurl I_{k,h}\lambda)_\LT\big]
    \notag\\
    &\hspace{30pt}+
    \big[a_h(\xi_{0,h},I_{k,h}\mu)-(\bcurl\xi_{0,h},\bcurl I_{k,h}\mu)_\LT\big]\\
    &\hspace{60pt}+
    \gamma^\frac12(\PiZ I_{k,h}\lambda-I_{k,h}\lambda,\xi_{0,h})_\LT
    -\gamma^\frac12(\zeta_{0,h},\PiZ I_{k,h}\mu-I_{k,h}\mu)_\LT\notag\\
    &\hspace{90pt}
    +\beta(\xi_{0,h},\PiZ I_{k,h}\mu-I_{k,h}\mu)_\LT.\notag
\end{align}
\par
 It follows from \eqref{eq:InterpolationOne},
 \eqref{eq:Inconsistency}, \eqref{eq:SystemDuality1} and \eqref{eq:SystemDuality5}  that
\begin{align}\label{eq:SystemDuality11}
  &\big|a_h(\xi_{0,h},I_{k,h}\mu)-(\bcurl\xi_{0,h},\bcurl I_{k,h}\mu)_\LT\big|\\
   &\hspace{50pt}
   \leq C_{\rm inc}|\xi_{0,h}|_\HOne|(\Id-\PiO)I_{k,h}\mu|_{h,1}
  \leq \CExpo \hfactor \|\bof\|_\LT\|\chi\|_\HOne.\notag
\end{align}
\par
 Similarly, we obtain
\begin{equation}\label{eq:SystemDuality12}
  \big|a_h(\xi_{0,h},I_{k,h}\mu)-(\bcurl\xi_{0,h},\bcurl I_{k,h}\mu)_\LT\big|
  \leq  \CExpo \hfactor \|\bof\|_\LT\|\chi\|_\HOne.
\end{equation}
\par
 Finally, we have
\begin{align}\label{eq:SystemDuality13}
 & |(\PiZ I_{k,h}\lambda-I_{k,h}\lambda,\xi_{0,h})_\LT|
 +|(\zeta_{0,h},\PiZ I_{k,h}\mu-I_{k,h}\mu)_\LT|\\
 &\hspace{30pt}
 |(\xi_{0,h},\PiZ I_{k,h}\mu-I_{k,h}\mu)_\LT|\leq
 \CExpo h^{1+\min(\ExpoStar,k)}
  \|\bof\|_\LT\|\chi\|_\HOne\notag
\end{align}
 by \eqref{eq:InterpolationZero}, \eqref{eq:SystemDuality1}
 and \eqref{eq:SystemDuality5}.
 \par
  The estimate \eqref{eq:SystemDualityZero} follows from
  \eqref{eq:SystemDuality6}--\eqref{eq:SystemDuality13}.

\par
 The derivation of \eqref{eq:SystemDualityOne} is very similar.
  First we observe that \eqref{eq:zetahxih1}, \eqref{eq:cAh} and
  \eqref{eq:Coercivity} imply
\begin{equation}\label{eq:SystemDuality15}
\|\zeta_{1,h}\|_\HOne+\|\xi_{1,h}\|_\HOne\leq C.
\end{equation}
 Then we   define
 $(\lambda,\mu)\in\HOne\times\HOnez$ by %
\begin{equation*}
    \cA\big((\psi,\eta),(\lambda,\mu)\big)+(\psi,1)_\LT(\lambda,1)_\LT
    =(\eta,\chi)_\LT
    \qquad\forall\,(\psi,\eta)\in\HOne\times\HOnez.
\end{equation*}
 The rest of the derivation is essentially identical to the derivation of
 \eqref{eq:SystemDualityZero}.
\end{proof}
%
\subsubsection{Estimate for $\xi_h$}\label{subsubsec:xihEstimatePositive} %
 We begin with $\xi_{0,h}$. Observe that the equation
\begin{equation}\label{eq:xihPositive1}
  (\bcurl \xi_0,\bcurl\eta)_\LT+  \beta(\xi_0,1)_\LT(\eta,1)_\LT=\gamma^\frac12
  (\zeta_0,\eta)_\LT     \qquad \forall\,\eta\in \HOnez
\end{equation}
 is part of the system \eqref{eq:System0}.
\par
 Let $\tilde\xi_0\in\HOnez$ be defined by the Dirichlet boundary value problem
\begin{equation}\label{eq:xihPositive2}
  (\bcurl\tilde\xi_0,\bcurl\eta)_\LT+ \beta(\tilde\xi_0,1)_\LT(\eta,1)_\LT
  =\gamma^\frac12(\zeta_{0,h},\eta)_\LT \qquad\forall\,\eta\in\HOnez.
\end{equation}
 Note that
\begin{equation}\label{eq:xihPositive3}
  \|\tilde\xi_0\|_\RegStar\leq \CExpo \|\bof\|_\LT
\end{equation}
 by elliptic regularity and \eqref{eq:SystemDuality1}.
\par
 It follows from \eqref{eq:SystemDualityZero}, \eqref{eq:xihPositive1} and
 \eqref{eq:xihPositive2}
 that
\begin{align*}
  &\big(\bcurl(\xi_0-\tilde\xi_0), \bcurl(\xi_0-\tilde\xi_0)\big)_\LT
  +\beta(\xi_0-\tilde\xi_0,1)_\LT (\xi_0-\tilde\xi_0,1)_\LT\\
  &\hspace{30pt}
  =\gamma^\frac12(\zeta_0-\zeta_{0,h},\xi_0-\tilde\xi_0)_\LT
  \leq \CExpo \hfactor\|\bof\|_\LT \|\xi_0-\tilde\xi_0\|_\HOne,
\end{align*}
 and hence we have, by \eqref{eq:PFDirichlet},
\begin{equation}\label{eq:xihPositive4}
  |\xi_0-\tilde\xi_0|_\HOne \leq \CExpo \hfactor\|\bof\|_\LT.
\end{equation}
\par
 On the other hand, we have the equation
\begin{equation}\label{eq:xihPositive5}
  a_h(\xi_{0,h},\eta_h)+\beta(\PiZ \xi_{0,h},\PiZ\eta_h)_\LT=
  \gamma^\frac12(\zeta_{0,h},\PiZ \eta_h)_\LT
  \qquad\forall\,\eta_h\in\zV_h^k
\end{equation}
 that is part of the system \eqref{eq:zetahxih0}.  Comparing \eqref{eq:xihPositive2}
 and \eqref{eq:xihPositive5},
 we see that $\xi_{0,h}\in\zV_h^k$ is the approximation of $\tilde\xi_0$ obtained by a
 virtual element method for the Dirichlet boundary value problem \eqref{eq:xihPositive2}
 whose right-hand
 side belongs to
 $\HOne$, and hence, because of \eqref{eq:xihPositive3},
\begin{equation}\label{eq:xihPositive6}
 |\tilde\xi_0-\xi_{0,h}|_\HOne+ |\tilde\xi_0-\PiO\xi_{0,h}|_{h,1}
 \leq \CExpo\hfactor \|\bof\|_\LT
\end{equation}
(cf. Appendix~\ref{append:VEMError}).
\par
 Combining \eqref{eq:xihPositive4}     and \eqref{eq:xihPositive6}, we find
\begin{equation}\label{eq:xihPositive7}
|\xi_0-\xi_{0,h}|_\HOne+  |\xi_0-\PiO\xi_{0,h}|_{h,1}\leq \CExpo\hfactor \|\bof\|_\LT.
\end{equation}
\par
 Next we consider $\xi_{1,h}$.  Observe that the  equation
\begin{equation}\label{eq:xihPositive8}
  (\bcurl\xi_1,\bcurl\eta)_\LT+\beta(\xi_1,\eta)_\LT=
  \gamma^\frac12(\zeta_1,\eta)_\LT+(1,\eta)_\LT
   \qquad\forall\,\eta\in \HOnez
\end{equation}
 is part of the system \eqref{eq:System1}.
\par
 Let $\tilde\xi_1\in\HOnez$ be defined by the Dirichlet boundary value problem
\begin{equation}\label{eq:xihPositive9}
  (\bcurl\tilde\xi_1,\bcurl\eta)_\LT+\beta(\tilde\xi_1,\eta)_\LT
  =\gamma^\frac12(\zeta_{1,h},\eta)_\LT+(1,\eta)_\LT
   \qquad\forall\,\eta\in \HOnez.
\end{equation}
  We have
\begin{equation}\label{eq:xihPositive10}
  \|\tilde\xi_1\|_\RegStar \leq \CExpo
\end{equation}
 by elliptic regularity and  \eqref{eq:SystemDuality15}.
\par
 The following analog of \eqref{eq:xihPositive4} is a consequence of
 \eqref{eq:SystemDualityOne}, \eqref{eq:xihPositive8} and \eqref{eq:xihPositive9}:
\begin{equation}\label{eq:xihPositive11} |\xi_1-\tilde\xi_1|_\HOne\leq \CExpo.
\end{equation}
\par
  Moreover, we also have the equation
\begin{equation}\label{eq:xihPositive12}
  a_h(\xi_{1,h},\eta_h)_\LT+\beta(\PiZ \xi_{1,h},\PiZ \eta_h)_\LT=
  \gamma^\frac12(\zeta_{1,h},\PiZ \eta_h)_\LT+(1,\eta_h)_\LT
\end{equation}
 for all $\eta_h\in\zV_h^k$
 that is part of the system \eqref{eq:zetahxih1}.
\par
 Comparing \eqref{eq:xihPositive9}  and \eqref{eq:xihPositive12}, we see that
 $\xi_{1,h}\in\zV_h^k$ is an approximation of $\tilde\xi_1$ obtained
 by a virtual element
 method for the Dirichlet boundary value problem \eqref{eq:xihPositive9} whose right-hand side
 belongs to $\HOne$.
   Therefore  we have, by \eqref{eq:xihPositive10},
\begin{equation}\label{eq:xihPositive13}
  |\tilde\xi_1-\xi_{1,h}|_\HOne+|\tilde\xi_1-\PiO\xi_{1,h}|_{h,1}
  \leq \CExpo\hfactor
\end{equation}
 (cf. Appendix~\ref{append:VEMError}).
\par
 Putting \eqref{eq:xihPositive11} and  \eqref{eq:xihPositive13} together, we obtain
\begin{equation}\label{eq:xihPositive14}
 |\xi_1-\xi_{1,h}|_\HOne+ |\xi_1-\PiO\xi_{1,h}|_{h,1}\leq \CExpo\hfactor.
\end{equation}
\par
 The estimate \eqref{eq:xihError}  follows from  \eqref{eq:SystemRepresentation},
 \eqref{eq:zetahxih}, \eqref{eq:xihPositive7}, \eqref{eq:xihPositive14}
 and the fact that
 $$|(1,\xi_{0})_\LT|\leq C\|\bof\|_\LT$$
 because of \eqref{eq:zeta0xi0Bdd}.
\section{Error Analysis of $\bu_h$} \label{sec:uhError}
 According to  \eqref{eq:uh1} and \eqref{eq:uh2}, we need to find error estimate
 for $\phi_h$, and in the case where $m\geq 1$, also error estimates for
 $\varphi_{j,h}$ and $c_j$ for $1\leq j\leq m$.
%
\subsection{Error Estimate for $\phi_h$}\label{subsec:phihError}
 We follow the strategy in Section~\ref{subsubsec:xihEstimateZero} and
 Section~\ref{subsubsec:xihEstimatePositive}.
\begin{lemma}\label{lem:phihError}
  Let $\phi_h$ be the approximation of $\phi$ defined by \eqref{eq:phih}.
  We have
\begin{equation} \label{eq:phihError}
  |\phi-\phi_h|_\HOne+|\phi-\PiO\phi_h|_{h,1}\leq \CExpo\hfactor\|\bof\|_\LT.
\end{equation}
\end{lemma}
\begin{proof}

 Let $\tilde\phi\in\HOne$ be defined by the Neumann boundary value problem
\begin{equation}\label{eq:phih1}
  (\bcurl\tilde\phi,\bcurl\psi)_\LT+(\tilde\phi,1)_\LT(\psi,1)_\LT
  =(\xi_h,\psi)_\LT\qquad
  \forall\,\psi\in \HOne.
\end{equation}
 Note that
\begin{equation}\label{eq:phih2}
\|\tilde\phi\|_\RegStar
  \leq \CExpo \|\bof\|_\LT
\end{equation}
 by elliptic regularity and \eqref{eq:xihBdd}.
\par
 It follows from \eqref{eq:PFDirichlet}, \eqref{eq:phiEq2},
 \eqref{eq:xihError} and \eqref{eq:phih1} that
\begin{align*}
  &\big(\bcurl(\phi-\tilde\phi),\bcurl(\phi-\tilde\phi)\big)_\LT
  +(\phi-\tilde\phi,1)_\LT(\phi-\tilde\phi,1)_\LT\\
  &\hspace{50pt}= (\xi-\xi_h,\phi-\tilde\phi)_\LT
  \leq \CExpo \hfactor\|\bof\|_\LT \|\phi-\tilde\phi\|_\LT,
  \notag
\end{align*}
 and hence we have, by \eqref{eq:PFNeumann},
\begin{equation}\label{eq:phih4}
  \|\phi-\tilde\phi\|_\HOne\leq \CExpo \hfactor \|\bof\|_\LT.
\end{equation}
\par
 Comparing \eqref{eq:phih} and \eqref{eq:phih1}, we see that
 $\phi_h\in V_h^k$ is the approximation
 of $\tilde\phi$ generated by a  virtual element method for the Neumann
 boundary value problem \eqref{eq:phih1} whose right-hand side belongs to
  $\HOne$, and hence we have,
  by \eqref{eq:phih2},
\begin{equation}\label{eq:phih5}
  |\tilde\phi-\phi_h|_\HOne+|\tilde\phi-\PiO\phi_h|_{h,1}\leq
  \CExpo \hfactor \|\bof\|_\LT
\end{equation}
 (cf. Appendix~\ref{append:VEMError}).
\par
 The estimate \eqref{eq:phihError} follows from \eqref{eq:phih4}
  and \eqref{eq:phih5}.
\end{proof}
\begin{corollary} \label{cor:phihBoundary}
  We have
\begin{equation} \label{eq:phihBoundary}
  \|\bn\times \bcurl(\PiO\phi_h)\|_{L^2(\p\O)}\leq \CExpo
  \hfactor h^{-\frac12}\|\bof\|_\LT.
\end{equation}
\end{corollary}
\begin{proof} Let $e$ be an edge of $D\in\cT_h$ such that $e\subset \p\O$ and
 $\delta>0$ be defined by
\begin{equation*}
  \delta=\frac{\ExpoStar}{2}-\frac14.
\end{equation*}
\par
    We have,
\begin{align}\label{eq:phihBoundary1}
& h_D^{-1}\|\bn\times\bcurl(\PiO\phi_h)\|_{L^2(e)}^2
    =h_D^{-1}\|\bn\times \bcurl(\PiO\phi_h-\phi)\|_{L^2(e)}^2\notag\\
    &\hspace{40pt}\leq h_D^{-1}\|\bcurl(\PiO\phi_h-\phi)\|_{L^2(e)}^2  \\
    &\hspace{40pt}\leq C_\delta\big[h_D^{-2}|\PiO\phi_h-\phi|_{H^1(D)}^2+
    h_D^{2\delta-1}|\PiO\phi_h-\phi|_{H^{\frac32+\delta}(D)}^2\big]  \notag
\end{align}
 by Remark~\ref{rem:phiBoundary} and a trace inequality with scaling
 (cf. \cite[Section~2.6]{BSung:2018:VEM}, where the arguments  can be
 extended to fractional order Sobolev spaces).
\par
 It follows from an inverse estimate (cf. \cite[Section~2.8]{BSung:2018:VEM},
 where the arguments   can also be extended to fractional
 order Sobolev spaces) that
\begin{align}\label{eq:phihBoundary2}
   |\PiO\phi_h-\phi|_{H^{\frac32+\delta}(D)}^2
   &\leq  2\big(|\PiO(\phi_h-\phi)|_{H^{\frac32+\delta}(D)}^2
   +|\PiO\phi-\phi|_{H^{\frac32+\delta}(D)}^2\big)  \\
   &\leq C_\delta \big(h_D^{-1-2\delta}|\Pi_{k,h}^1(\phi_h-\phi)|_{H^1(D)}^2
   + |\PiO\phi-\phi|_{H^{\frac32+\delta}(D)}^2\big).
\notag
\end{align}
\par
 We obtain the estimate \eqref{eq:phihBoundary} by combining \eqref{eq:PiOPythagoras},
 \eqref{eq:PiOErrorEstimate},
 \eqref{eq:phihError},  \eqref{eq:phihBoundary1}, \eqref{eq:phihBoundary2}
 and the elliptic regularity estimate
  $$\|\phi\|_\RegStar\leq \CExpo \|\bof\|_\LT$$
 that follows from \eqref{eq:phiEq2} and \eqref{eq:xiBddPositive}.
\end{proof}

\subsection{Error Estimate for $\varphi_{j,h}$ and $c_{j,h}$
for a Multiply Connected Domain}
\label{subsec:harmonicError}
 Let  $\O$ be a domain with Betti number $m\geq1$.
\subsubsection{Error Analysis of $\varphi_{i,h}$} \label{subsubsec:HamonicError}
Let the discrete harmonic functions
 $\varphi_{1,h},\ldots,\varphi_{m,h}\in V_h$ be defined by
 \eqref{subeqs:DHarmonicFunctions}.
\begin{lemma}\label{lem:varphi}
 We have
\begin{alignat}{3}
  |\varphi_j-\varphi_{j,h}|_\HOne+|\varphi_j-\PiO\varphi_{j,h}|_{h,1}&\leq
   \CExpo\hfactor &\qquad&\text{for}\quad 1\leq j\leq m,\label{eq:HarmonicErrors} \\
   |\varphi_{j,h}-\PiO\varphi_{j,h}|_{h,1}&\leq
   \CExpo\hfactor &\qquad&\text{for}\quad 1\leq j\leq m.\label{eq:DHarmonicErrors}
\end{alignat}
\end{lemma}
\begin{proof}
 Since
 $\varphi_{j,h}$  is an approximation of
  the harmonic function $\varphi_j$
 obtained by a virtual element method for the Dirichlet boundary value problem
 \eqref{subeqs:HarmonicFunctions} whose right-hand side
 belongs to $\HOne$,
  the estimate \eqref{eq:HarmonicErrors} follows from
  \eqref{subeqs:HarmonicRegularity} (cf. Appendix~\ref{append:VEMError}).
\par
 We then deduce from \eqref{eq:PiOPythagoras}, \eqref{eq:PiOErrorEstimate}
 and \eqref{eq:HarmonicErrors}
 that
\begin{align*}
  |\varphi_{j,h}-\PiO\varphi_{j,h}|_{h,1}
  &\leq |(\varphi_{j,h}-\varphi_j)-\PiO(\varphi_{j,h}-\varphi_j)|_{h,1}
  +|\varphi_j-\PiO\varphi_j|_{h,1}    \notag\\
  &\leq |\varphi_{j,h}-\varphi_j|_{h,1}+|\varphi_j-\PiO\varphi_{j,h}|_{h,1} \\
  &\leq \CExpo \hfactor.
\end{align*}
\end{proof}
%
\subsubsection{Error Analysis of $c_{j,h}$}\label{subsubsec:CoeffiicentError}
 We begin by comparing the
 right-hand sides of the  systems
 \eqref{eq:Coefficients} and \eqref{eq:DCoefficients}: %
\begin{equation}\label{eq:RHS}
  |\gamma^{-1}(\bof,\grad\varphi_{i,h})_\LT-\gamma^{-1}(\bof,\grad\PiO\varphi_{i,h})_\LT|
  \leq \CExpo\hfactor\|\bof\|_\LT,
\end{equation}
  which  follows from \eqref{eq:DHarmonicErrors}.
 \par
 Now we compare the components of the  matrices in the systems
  \eqref{eq:Coefficients} and \eqref{eq:DCoefficients}.
\par
 Since $\varphi_i-\varphi_{i,h}$ belongs to $\HOnez$, it follows from \eqref{eq:Harmonic}
 that
\begin{align*}
  &(\grad\varphi_i,\grad\varphi_j)_\LT-(\grad\varphi_{i,h},\grad\varphi_{j,h})_\LT \\
      &\hspace{40pt}=
      \big(\grad(\varphi_i-\varphi_{i,h}),\grad(\varphi_{j,h}-\varphi_j)\big)_\LT
      \qquad\text{for}\quad 1\leq i,j\leq m,
      \notag
\end{align*}
 and hence   we have
\begin{equation}\label{eq:Component1}
  \big|(\grad\varphi_i,\grad\varphi_j)_\LT
  -(\grad\varphi_{i,h},\grad\varphi_{j,h})_\LT\big|
  \leq \CExpo h^{2\min(\ExpoStar,k)} \end{equation}
 by \eqref{eq:HarmonicErrors}.
\par
 Next we obtain from \eqref{eq:Identity} and Remark~\ref{rem:Grad} that
\begin{align*}
  a_h(\varphi_{i,h},\varphi_{j,h})&
  =(\grad\varphi_{i,h},\grad\varphi_{j,h})_\LT
-\sum_{D\in\cT_h} \big(\grad\varphi_{i,h},\grad(\Id-\Pi_{k,D}^1)\varphi_{j,h}\big)_\LTD\\
  &\hspace{30pt}
  +\sum_{D\in\cT_h}S^D\big((\Id-\Pi_{k,D}^1)\varphi_{i,h},(\Id-\Pi_{k,D}^1)\varphi_{j,h}\big)
\end{align*}
 which implies, through \eqref{eq:SDBound}, \eqref{eq:HarmonicErrors}
  and \eqref{eq:DHarmonicErrors},
\begin{equation} \label{eq:Component2}
  \big|(\grad\varphi_{i,h},\grad\varphi_{j,h})_\LT-a_h(\varphi_{i,h},\varphi_{j,h})|
  \leq \CExpo \hfactor \qquad\text{for}\quad 1\leq i,j\leq m.
\end{equation}
 \par
 Putting \eqref{eq:Component1} and \eqref{eq:Component2} together, we find
\begin{equation}\label{eq:ComponentError}
  \big|(\grad\varphi_i,\grad\varphi_j)_\LT-a_h(\varphi_{i,h},\varphi_{j,h})\big|
  \leq \CExpo \hfactor \qquad\text{for}\quad 1\leq i,j\leq m.
\end{equation}
\begin{lemma} \label{lem:CoefficientError}
  We have
\begin{equation}\label{eq:CoefficientError}
  |c_j-c_{j,h}|\leq \CExpo\hfactor\|\bof\|_\LT.
\end{equation}
\end{lemma}
\begin{proof}
 We can write \eqref{eq:Coefficients} in the form of
\begin{equation*}
  \bA \bc=\bb,
\end{equation*}
 where $\bA\in\R^{m\times m}$ is the matrix with components
 $\bA(i,j)=(\grad\varphi_i,\grad\varphi_j)_\LT$,
 $\bc\in\R^m$ is the vector with components $\bc(i)=c_i$, and
 $\bb\in\R^m$ is the vector with
 components $\bb(i)=\gamma^{-1}(\bof,\grad\varphi_i)_\LT$.
\par
 Note that the $\HOne$ norms of the harmonic functions $\varphi_1,\ldots,\varphi_m$
  defined by \eqref{subeqs:HarmonicFunctions} are bounded,
 and hence
\begin{equation}\label{eq:bbBdd}
  \|\bb\|_\infty\leq C\|\bof\|_\LT.
\end{equation}
\par
 Similarly we can write \eqref{eq:DCoefficients} as
\begin{equation*}
  \bA_h\bc_h=\bb_h,
\end{equation*}
 where $\bA_h\in\R^{m\times m}$ is the matrix with components
 $\bA_h(i,j)=a_h(\varphi_{i,h},\varphi_{j,h})$,
 $\bc_h\in\R^m$ is the vector with components $\bc_h(i)=c_{i,h}$,
 and $\bb_h\in\R^m$ is the vector with
 components $\bb_h(i)=\gamma^{-1}(\bof,\grad\PiO\varphi_{i,h})_\LT$.
\par
 The estimates \eqref{eq:RHS} and \eqref{eq:ComponentError} are translated into
 the estimates
\begin{align}
  \|\bb-\bb_h\|_\infty&\leq \CExpo \hfactor\|\bof\|_\LT\label{eq:bbMinusbbh}
  \intertext{and}
  \|\bA-\bA_h\|_\infty&\leq \CExpo \hfactor.\label{eq:bAMinusBAh}
\end{align}
\par
 The estimate \eqref{eq:CoefficientError} then follows from
 \eqref{eq:bbBdd}--\eqref{eq:bAMinusBAh}
 and the identity
\begin{equation*}
  \bc-\bc_h=\bA^{-1}(\bb-\bb_h)+\bA^{-1}(\bA_h-\bA)\bA_h^{-1}\bb_h.
 \end{equation*}
\end{proof}

\subsection{Error Estimates for $\bu_h$}\label{subsec:uhError}
 We can put the results in Section~\ref{subsec:phihError}
 and Section~\ref{subsec:harmonicError}
 together to obtain error estimates for $\bu_h$.
\begin{theorem}\label{thm:buhError}
  Let $\omega$ be the maximum interior angle of $\O$ and
  $\epsilon$ be an arbitrary positive number.
    We have
\begin{subequations}\label{subeqs:buhError}
\begin{align}
  h^\frac12 \|\bn\times\bu_h\|_{L^2(\p\O)}+\|\bu-\bu_h\|_\LT&\leq
   \CO \hfactorG\|\bof\|_\LT\label{eq:buhErrorG}
   \intertext{if $(\pi/\omega)$ is not an integer, and}
  h^\frac12 \|\bn\times\bu_h\|_{L^2(\p\O)}+\|\bu-\bu_h\|_\LT&\leq
   \COE \hfactorE\|\bof\|_\LT\label{eq:buhErrorE}
\end{align}
 if $(\pi/\omega)$ is an integer.
\end{subequations}
\end{theorem}
\begin{proof}
  In the case where $\gamma=0$, the estimates \eqref{subeqs:buhError}
  follow from \eqref{eq:HodgeDecomposition}, \eqref{eq:uh1} ,
  Lemma~\ref{lem:phihError} and
  Corollary~\ref{cor:phihBoundary}.
\par
 In the case where $m\geq 1$, we use \eqref{eq:CoefficientBdd},
 \eqref{eq:uh2} and also Lemma~\ref{lem:varphi}
 and Lemma~\ref{lem:CoefficientError}.
\end{proof}
%
\section{Numerical Results}\label{sec:Numerics}
 We use both structured and unstructured Voronoi meshes for the
 first experiment on the unit square where the exact solution is available.
 The structured meshes are dual meshes of structured simplicial meshes.
  The unstructured meshes are generated by
 \texttt{PolyMesher} (cf. \cite{TPPM:2012:PolyMesher}) with random seed points.
 Examples of these meshes are provided in Figure~\ref{fig:Meshes}.

\begin{figure}[H]
\centering
    \includegraphics[width=1.5in]{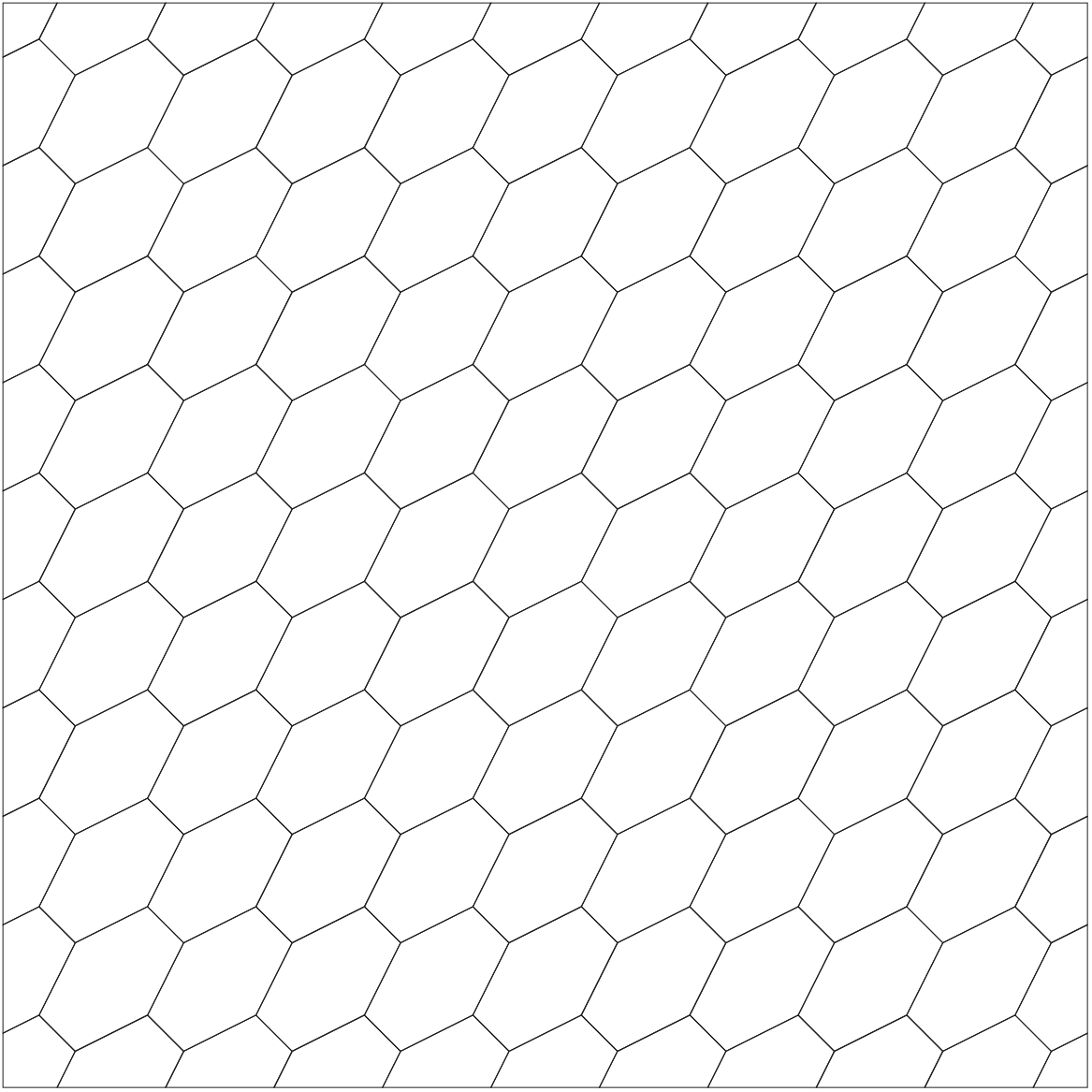}
\qquad
    \includegraphics[width=1.5in]{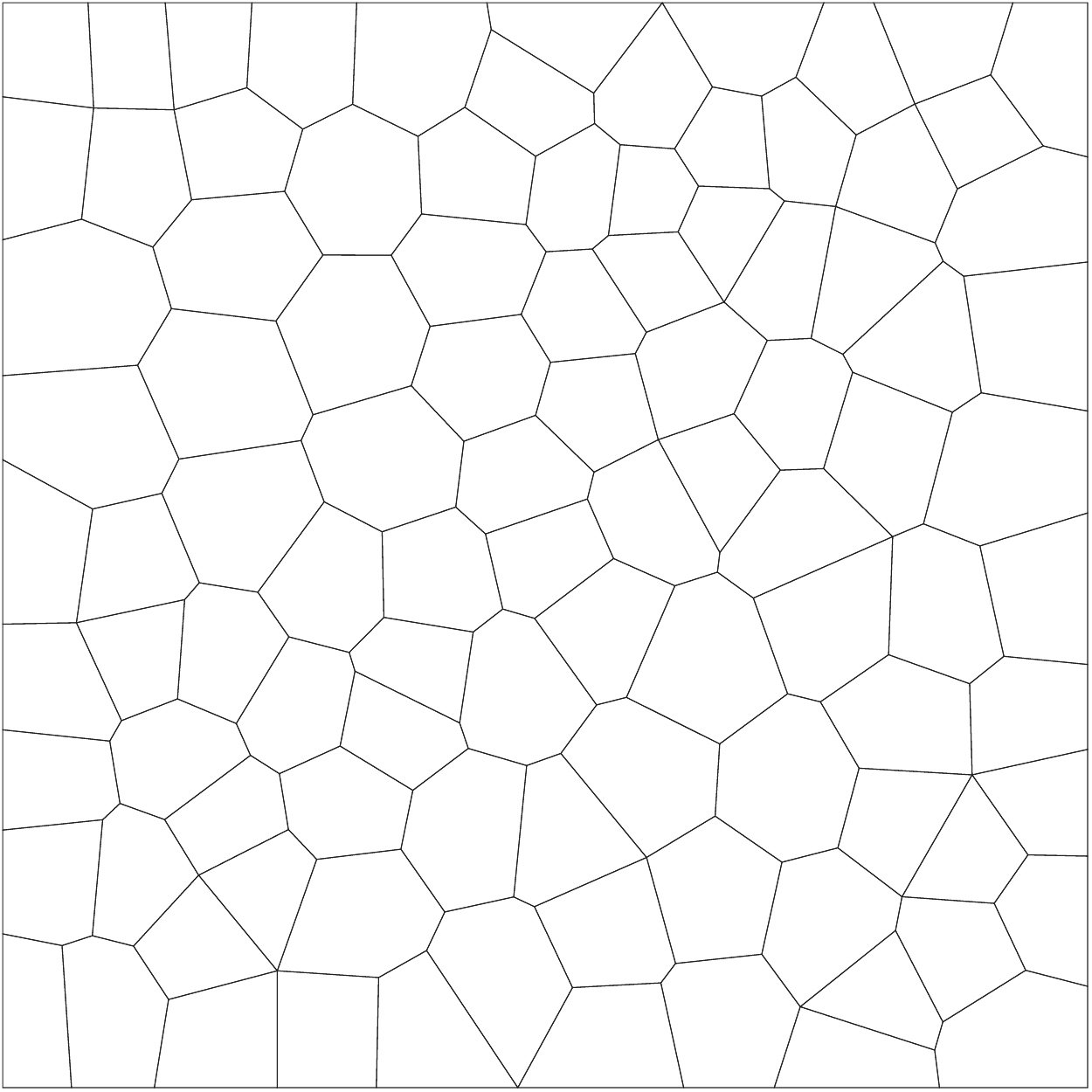}
\caption{Structured Voronoi mesh (left) and unstructured Voronoi mesh (right)}
\label{fig:Meshes}
\end{figure}
\par
 For the other experiments where the exact solutions are not available,
 we use nested polytopal meshes to facilitate the computation of the differences
 between the solutions on consecutive
  levels.  These nested meshes are
 obtained by the following procedure. (i) We create a random initial Voronoi mesh
 on the domain $\O$. (ii) We connect the barycenter of
 any polygon in the initial mesh with the midpoints of its edges
 (cf. Figure~\ref{fig:Refinements}, left)
 to form a quadrilateral mesh of $\O$.  (iii) We uniformly refine the
 quadrilateral mesh by connecting the intersection of the two diagonals
 to the midpoints of its edges
  (cf. Figure~\ref{fig:Refinements}, right).
 It was shown in \cite[Proposition~5.2]{DOS:1994:Hierachical} that
 the meshes generated by this procedure satisfy
  the shape regularity assumptions in \eqref{subeqs:MA}.
\begin{figure}[H]
\centering
\includegraphics[height=1in]{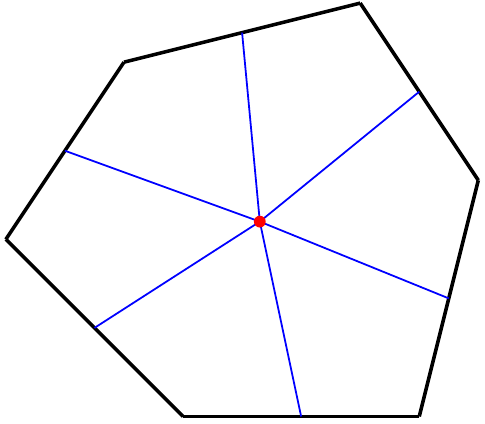}\qquad\qquad
\includegraphics[height=1in]{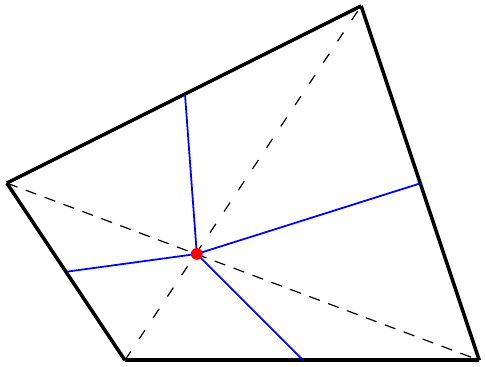}
\caption{Refinement procedure: refining a polygon into quadrilaterals (left) and
uniform refinement of a quadrilateral (right).}
\label{fig:Refinements}
\end{figure}
%
\par\noindent
{\bf Experiment 1} 
\quad
 We solve \eqref{eq:QuadCurl} on the unit square $(0,1)^2$ with $\beta=\gamma=0$, and we use
 both structured and nonstrutured Voronoi meshes.
 The manufactured solution is given by
  $\bu = \bcurl{\phi}$, where $\phi(x) = \sin^3(\pi x_1) \sin^3(\pi x_2)$. The numerical results
   for the virtual element methods are presented in Table~\ref{table:SmoothSolutionStructured}
 and Table~\ref{table:SmoothSolutionUnstructured}, where
  \begin{align*}
        e_{\bu_h}= \| \bu - \bu_h\|_\LT, \;
        e_{\xi_h} = |\xi - \PiO\xi_h|_{h,1} \quad (\xi=\curl\bu)
        \quad\text{and}\quad \eBdry=\|\bn\times\bu_h\|_{L^2(\p\O)}.
    \end{align*}
\par
 It is observed that, for both types of meshes,
  $\|\bu-\bu_h\|_\LT$ and $|\curl\bu-\PiO\xi_h|_{h,1}$ are $O(h)$ for the
 virtual element method with $k=1$ and $O(h^2)$ if $k=2$.  They agree with the results in
 Theorem~\ref{thm:xihError} and Theorem~\ref{thm:buhError} where $\omega=\pi/2$.
\par
 On the other hand, the $O(h^2)$ behavior of $\|\bn\times\bu_h\|_{L^2(\p\O)}$ for
 both $k=1$ and $k=2$ is better than the estimates in Theorem~\ref{thm:buhError}.
 In order to understand this behavior,
 we checked the maximum norm error $\max_{p\in\cN_{\p\O}}|(\bn\times\bu_h)(p)|$, where $\cN_{\p\O}$ is the
 set of the nodes of the virtual element space $V_h^k$
 along $\p\O$, and found that they
 also  decrease to $0$ at the rate of $O(h^2)$.  Therefore a possible explanation is that for problems with
 sufficiently smooth solutions the virtual element methods may actually converge in the maximum norm and
 may even enjoy some super-convergence when $k=1$, but such  analyses
 are beyond the scope of this paper.
\begin{longtable}[htbp]{@{}|c|c|c|c|c|c|c|c|@{}}
  \caption{Errors and convergence rates for Experiment~1
   on non-nested structured Voronoi meshes.}
  \label{table:SmoothSolutionStructured}\\
  \hline
  $h$ & $\#{\texttt{dofs}}$ & $e_{\bu_h}$ & $\texttt{rate}$
  & $e_{\xi_h}$ & $\;\; \texttt{rate}$ & $\eBdry$ & $\texttt{rate}$ \\ \hline
  \multicolumn{8}{|c|}{$k = 1$} \\ \hline
  2.9814e-01 & 90    & 1.1647e+00 & -      & 9.9246e+01 & -      & 2.6320e-01 & -      \\
  1.4907e-01 & 280   & 5.4917e-01 & 1.0847 & 5.1436e+01 & 0.9482 & 8.9737e-02 & 1.5524 \\
  7.4536e-02 & 960   & 2.4671e-01 & 1.1544 & 2.5104e+01 & 1.0349 & 2.3093e-02 & 1.9583 \\
  3.7268e-02 & 3520  & 1.1807e-01 & 1.0632 & 1.2425e+01 & 1.0146 & 5.6958e-03 & 2.0195 \\
  1.8634e-02 & 13440 & 5.8297e-02 & 1.0182 & 6.1970e+00 & 1.0036 & 1.4068e-03 & 2.0175 \\
  9.3169e-03 & 52480 & 2.9054e-02 & 1.0047 & 3.0971e+00 & 1.0007 & 3.4917e-04 & 2.0104 \\ \hline

  \multicolumn{8}{|c|}{$k = 2$} \\ \hline
  2.9814e-01 & 251    & 2.7456e-01 & -      & 3.0861e+01 & -      & 1.8833e-01 & -      \\
  1.4907e-01 & 801    & 7.1283e-02 & 1.9455 & 8.6428e+00 & 1.8362 & 4.9871e-02 & 1.9170 \\
  7.4536e-02 & 2801   & 1.8367e-02 & 1.9565 & 2.2548e+00 & 1.9385 & 1.2614e-02 & 1.9832 \\
  3.7268e-02 & 10401  & 4.6543e-03 & 1.9805 & 5.7301e-01 & 1.9764 & 3.1624e-03 & 1.9959 \\
  1.8634e-02 & 40001  & 1.1703e-03 & 1.9916 & 1.4422e-01 & 1.9903 & 7.9115e-04 & 1.9990 \\
  9.3169e-03 & 156801 & 2.9336e-04 & 1.9962 & 3.6163e-02 & 1.9957 & 1.9782e-04 & 1.9997 \\ \hline
\end{longtable}
\begin{longtable}{@{}|c|c|c|c|c|c|c|c|@{}}
  \caption{Errors and  convergence rates for Experiment~1
   on non-nested unstructured  Voronoi meshes.}
  \label{table:SmoothSolutionUnstructured}\\
  \hline
  $h$ & $\#{\texttt{dofs}}$ & $e_{\bu_h}$ & $\texttt{rate}$ & $e_{\xi_h}$ &
  $\texttt{rate}$ & $\eBdry$ & $\texttt{rate}$ \\ \hline
  \multicolumn{8}{|c|}{$k = 1$} \\ \hline
  3.7637e-01 & 48     & 1.2066e+00 & -      & 1.0267e+02 & -      & 5.2892e-01 & -      \\
  1.8360e-01 & 184    & 5.3903e-01 & 1.1225 & 5.1563e+01 & 0.9594 & 2.9837e-01 & 0.7976 \\
  8.4283e-02 & 1234   & 1.8092e-01 & 1.4022 & 1.8905e+01 & 1.2887 & 6.0540e-02 & 2.0486 \\
  4.1405e-02 & 5947   & 8.1750e-02 & 1.1177 & 8.6618e+00 & 1.0981 & 1.1904e-02 & 2.2883 \\
  2.0988e-02 & 26477  & 3.8903e-02 & 1.0929 & 4.1494e+00 & 1.0831 & 2.6990e-03 & 2.1840 \\
  1.0757e-02 & 96734  & 2.0202e-02 & 0.9805 & 2.1549e+00 & 0.9803 & 7.5059e-04 & 1.9148 \\ \hline

  \multicolumn{8}{|c|}{$k = 2$} \\ \hline
  3.7637e-01 & 145     & 3.0547e-01 & -      & 3.4296e+01 & -      & 4.4138e-01 & -      \\
  1.8360e-01 & 567     & 7.2817e-02 & 1.9975 & 8.8916e+00 & 1.8806 & 1.7322e-01 & 1.3030 \\
  8.4283e-02 & 3867    & 1.1101e-02 & 2.4158 & 1.3485e+00 & 2.4225 & 3.1320e-02 & 2.1967 \\
  4.1405e-02 & 18693   & 2.3647e-03 & 2.1757 & 2.8657e-01 & 2.1790 & 6.2542e-03 & 2.2666 \\
  2.0988e-02 & 82953   & 5.2424e-04 & 2.2171 & 6.3400e-02 & 2.2201 & 1.4410e-03 & 2.1604 \\
  1.0757e-02 & 303467  & 1.4262e-04 & 1.9477 & 1.7275e-02 & 1.9453 & 4.0201e-04 & 1.9101 \\\hline
\end{longtable}
\par\noindent

%
\par\noindent
{\bf Experiment~2} \quad We solve \eqref{eq:QuadCurl} on the unit square $(0,1)^2$ with $\beta=\gamma=0$,
 a piecewise constant right-hand side
  \begin{equation}\label{eq:PCRHS}
  \bof=\begin{cases}
    [1/4, 5/4]^t &\qquad\text{if $|x|<2^{-\frac12}$},\\[2pt]
    [1/2,3/2]^t&\qquad\text{if $2^{-\frac12}\leq|x|<1$},\\[2pt]
    [1,2]^t&\qquad\text{if $|x|\geq 1$},
  \end{cases}
  \end{equation}
 and we use nested meshes obtained by refining a randomly generated initial Voronoi mesh
 (cf. Figure~\ref{fig:Square}).
\begin{figure}[htbp]
 \centering
 \includegraphics[width=1.2in]{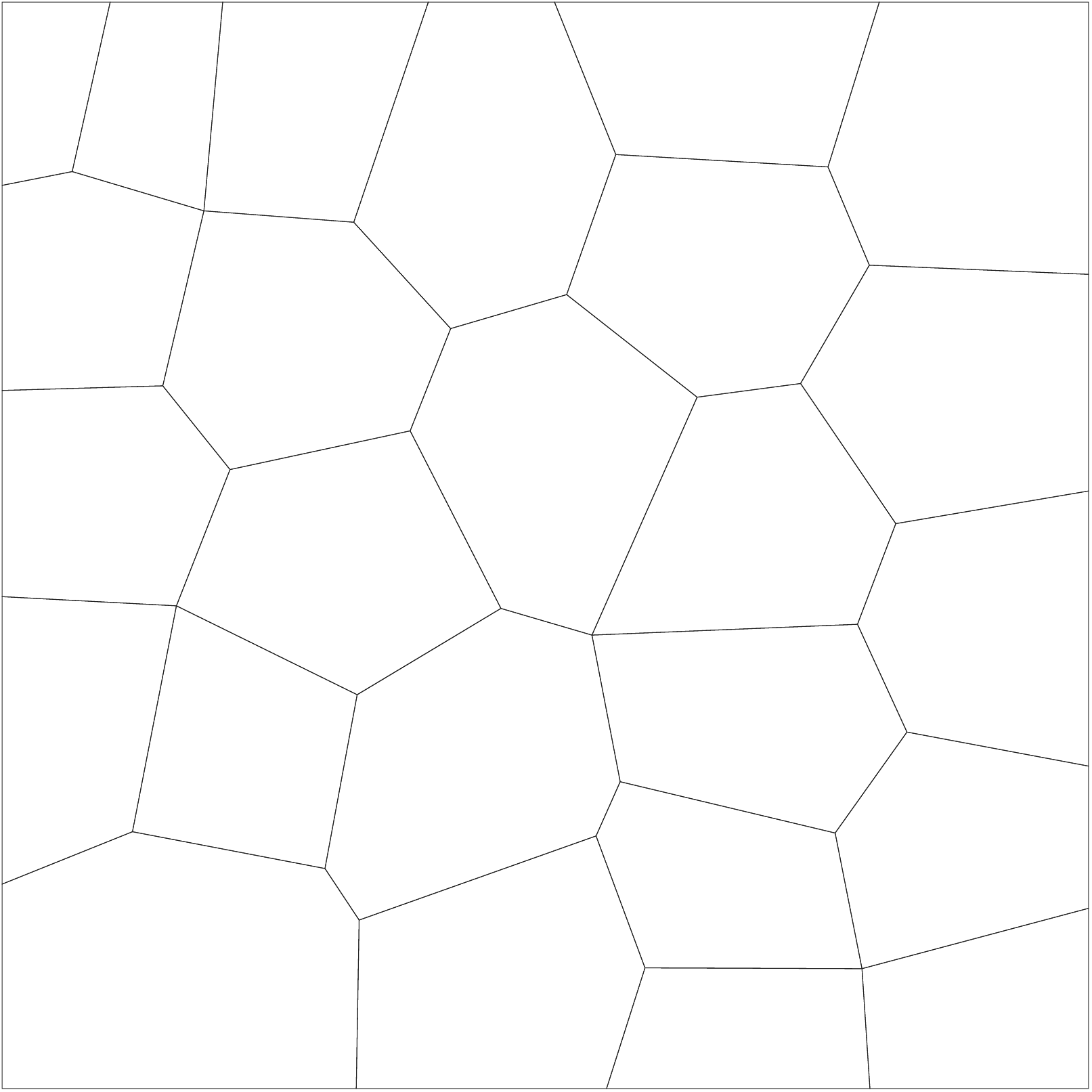}\qquad
 \includegraphics[width=1.2in]{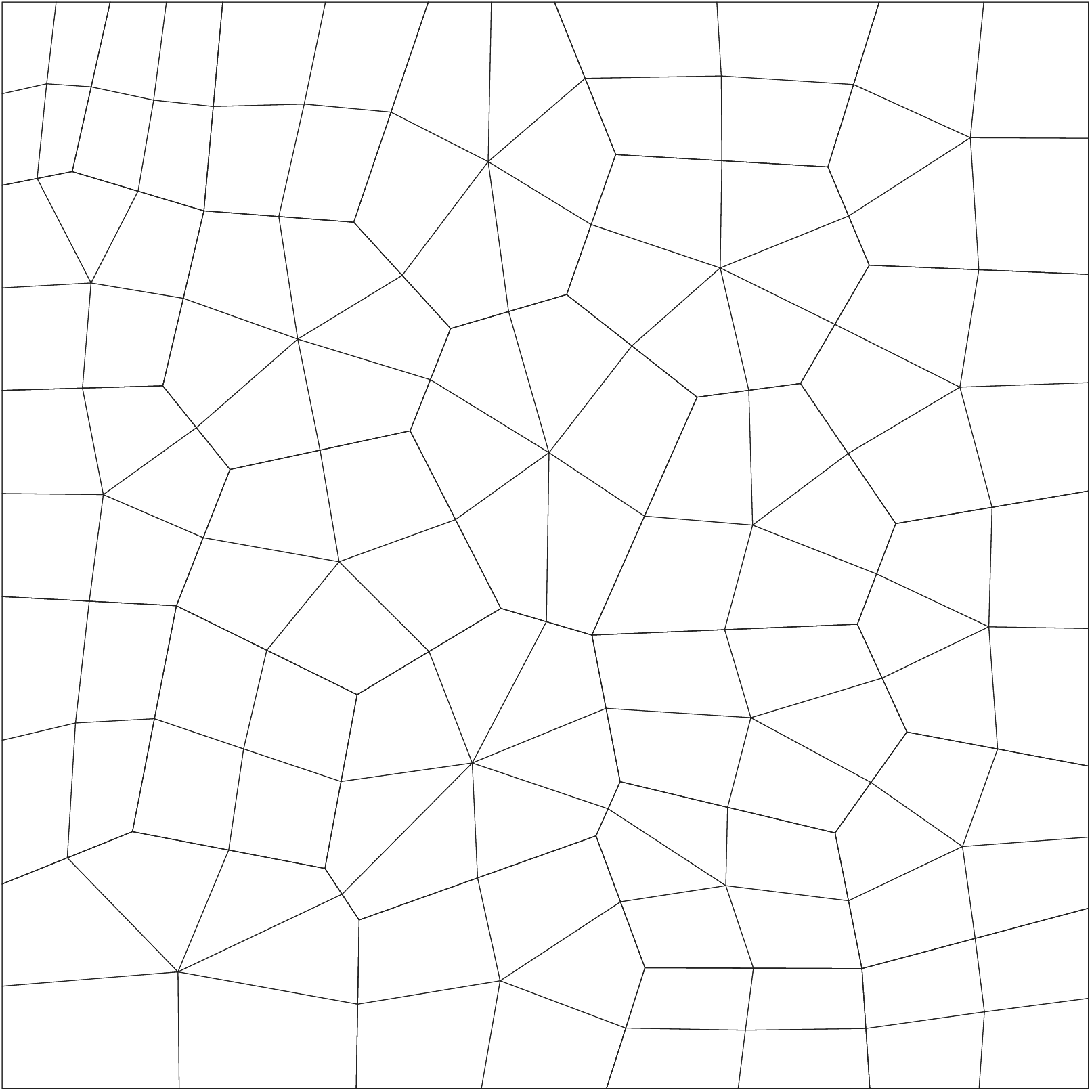}\qquad
 \includegraphics[width=1.2in]{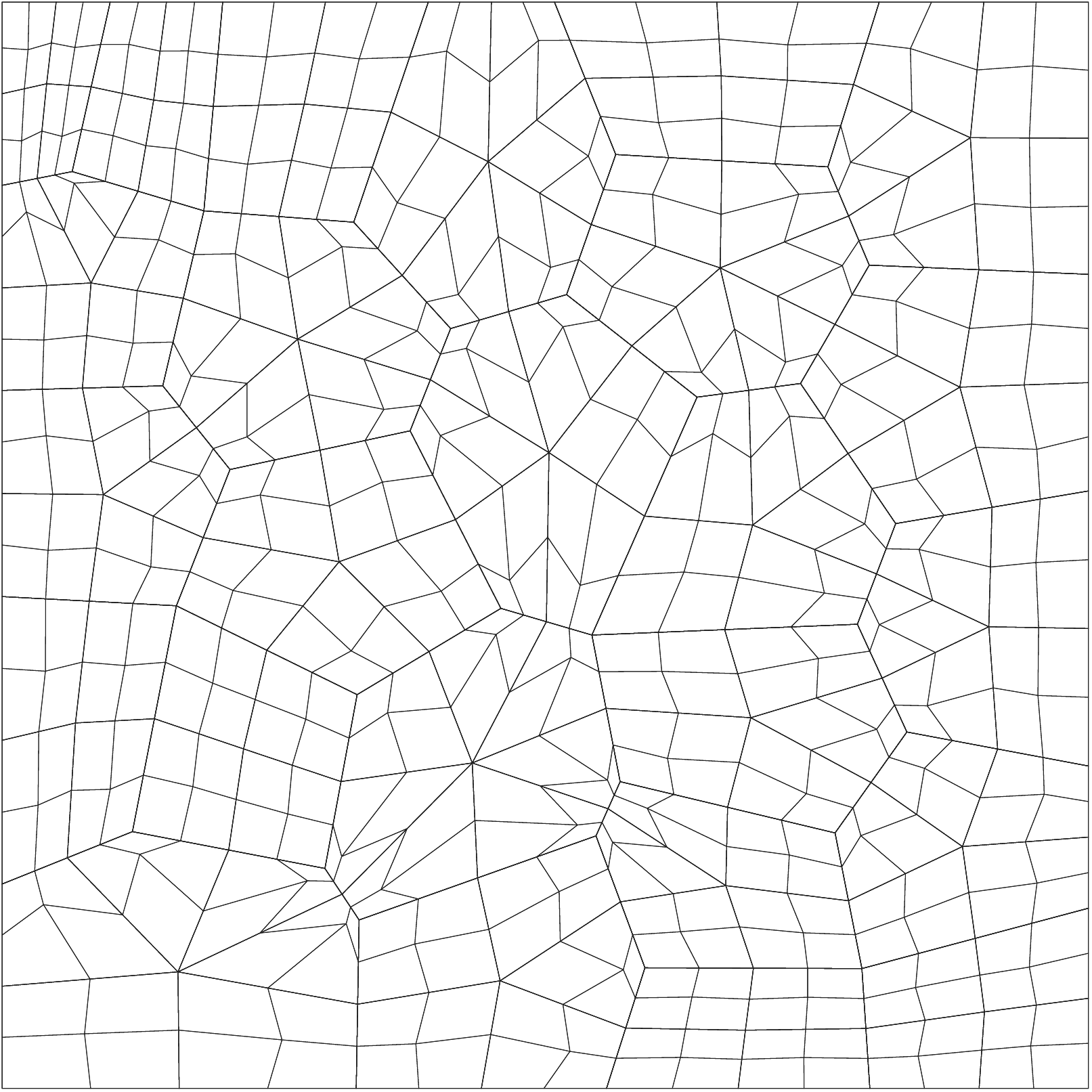}
\caption{Nested meshes on $\O=(0,1)^2$.}
\label{fig:Square}
\end{figure}
\par
 Since the exact solution is not available, we displayed in
 Table~\ref{table:PiecewiseSmoothRHSUnitSquare}
  the relative errors computed by
\begin{align*}
        \texttt{rel} \; e_{\bu_h}^i =
        \frac{\|\bu_{h}^i - \bu_{h}^{i+1}\|_\LT}{\|\bu_{h}^{i+1}\|_\LT}
        \quad \mbox{and} \quad \texttt{rel} \; e_{\xi_h}^i = \frac{|\Pi^{1,i}_{k,h} \xi_h^{i} - \Pi^{1,i+1}_{k,h} \xi_h^{i+1}|_{h_{i+1},1}}{|\Pi^{1,i+1}_{k,h} \xi_h^{i+1}|_{h_{i+1},1}},
\end{align*}
 where $i$ denotes the mesh level.
\begin{longtable}{@{}|c|c|c|c|c|c|c|c|@{}}
  \caption{Errors and convergence rates for Experiment~2
   on nested meshes.}
  \label{table:PiecewiseSmoothRHSUnitSquare}\\
  \hline
  $h$ & $\#{\texttt{dofs}}$ & $\texttt{rel}\;e_{\bu_h}$ & $\texttt{rate}$ & $\texttt{rel}\;e_{\xi_h}$
   & $\texttt{rate}$ & $\eBdry$ & $\texttt{rate}$ \\ \hline
  \multicolumn{8}{|c|}{$k = 1$} \\ \hline
  3.7637e-01 & 48     & -          & -      & -          & -      & - & -      \\
  1.9603e-01 & 145    & 2.3899e-01 & -      & 4.4218e-01 & -      & 3.1270e-04 & -      \\
  1.6110e-01 & 537    & 9.4817e-02 & 4.7112 & 2.1058e-01 & 3.7805 & 1.6109e-04 & 3.3801 \\
  8.0550e-02 & 2065   & 5.0310e-02 & 0.9143 & 1.1293e-01 & 0.8990 & 8.2304e-05 & 0.9688 \\
  4.0275e-02 & 8097   & 2.5513e-02 & 0.9796 & 5.7736e-02 & 0.9679 & 4.1834e-05 & 0.9763 \\
  2.0137e-02 & 32065  & 1.2801e-02 & 0.9950 & 2.9098e-02 & 0.9885 & 2.1120e-05 & 0.9861 \\
  1.0069e-02 & 127617 & 6.4186e-03 & 0.9959 & 1.4598e-02 & 0.9952 & 1.0616e-05 & 0.9923 \\
  5.0344e-03 & 509185 & 3.2129e-03 & 0.9984 & 7.3098e-03 & 0.9979 & 5.3227e-06 & 0.9961 \\ \hline

  \multicolumn{8}{|c|}{$k = 2$} \\ \hline
  3.7637e-01 & 145     & -          & -      & -          & -      & - & -      \\
  1.9603e-01 & 537     & 3.7070e-02 & -      & 1.1343e-01 & -      & 9.2194e-05 & -      \\
  1.6110e-01 & 2065    & 9.0838e-03 & 7.1667 & 3.1296e-02 & 6.5621 & 2.6338e-05 & 6.3848 \\
  8.0550e-02 & 8097    & 2.7673e-03 & 1.7148 & 9.6154e-03 & 1.7025 & 7.0165e-06 & 1.9083 \\
  4.0275e-02 & 32065   & 7.4310e-04 & 1.8969 & 2.7227e-03 & 1.8203 & 1.8103e-06 & 1.9545 \\
  2.0137e-02 & 127617  & 1.9973e-04 & 1.8955 & 7.4779e-04 & 1.8644 & 4.5982e-07 & 1.9771 \\
  1.0069e-02 & 509185  & 5.2869e-05 & 1.9176 & 2.0225e-04 & 1.8865 & 1.1588e-07 & 1.9885 \\
  5.0344e-03 & 2034177 & 1.2940e-05 & 2.0306 & 5.3856e-05 & 1.9089 & 2.9086e-08 & 1.9942 \\ \hline
\end{longtable}
\par
 Again the observed $O(h)$ ($k=1$) and $O(h^2)$ ($k=2$) behaviors
 for the estimated errors $\|\bu-\bu_h\|_\LT$
 and $|\curl \bu-\PiO\xi_h|_{h,1}$ agree with Theorem~\ref{thm:xihError}
  and Theorem~\ref{thm:buhError}.
\par
 Since the solution $\bu$ is no longer $C^\infty$, there is a difference in the behavior of
 $\|\bn\times\bu_h\|_{L^2(\p\O)}$ for $k=1$ ($O(h)$) and $k=2$ ($O(h^2)$).
 Both are better than the estimate in Theorem~\ref{thm:buhError}. Again the maximum norm error
 $\max_{p\in\cN_{\p\O}}|(\bn\times\bu_h)(p)|$ exhibits the same behaviors.
\par\smallskip\noindent
{\bf Experiment 3} \quad We solve \eqref{eq:QuadCurl} on the $\Gamma$-shaped domain
$\O=(-1,1)^2\setminus\big([0,1]\times[-1,0]\big)$ with $\beta=\gamma=0$,
 the piecewise constant right-hand side $\bof$
in \eqref{eq:PCRHS}, and we use
 nested meshes obtained by refining a randomly generated initial Voronoi mesh
(cf. Figure~\ref{fig:Gamma}).
\begin{figure}[htbp]
 \centering
 \includegraphics[width=1.2in]{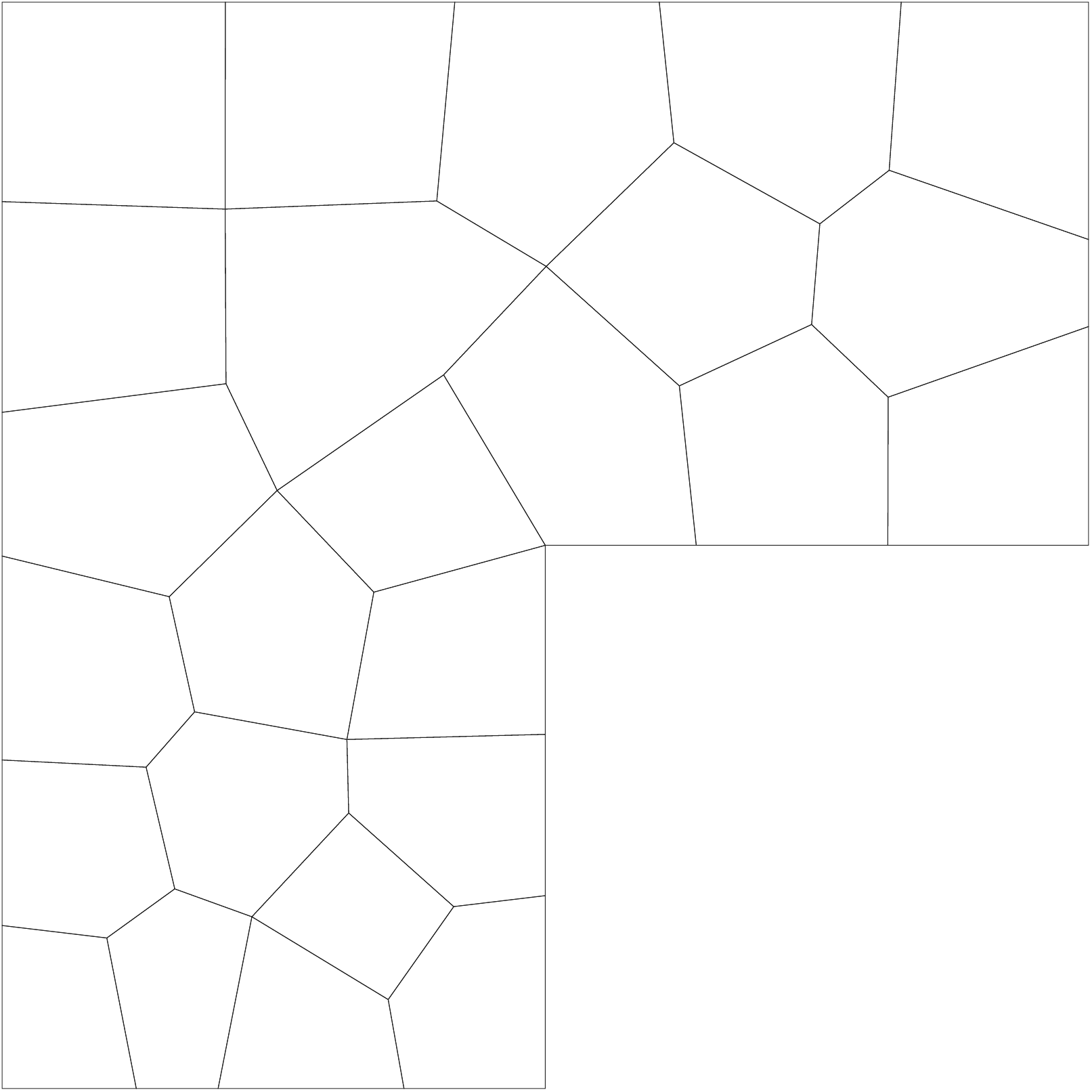}\qquad
 \includegraphics[width=1.2in]{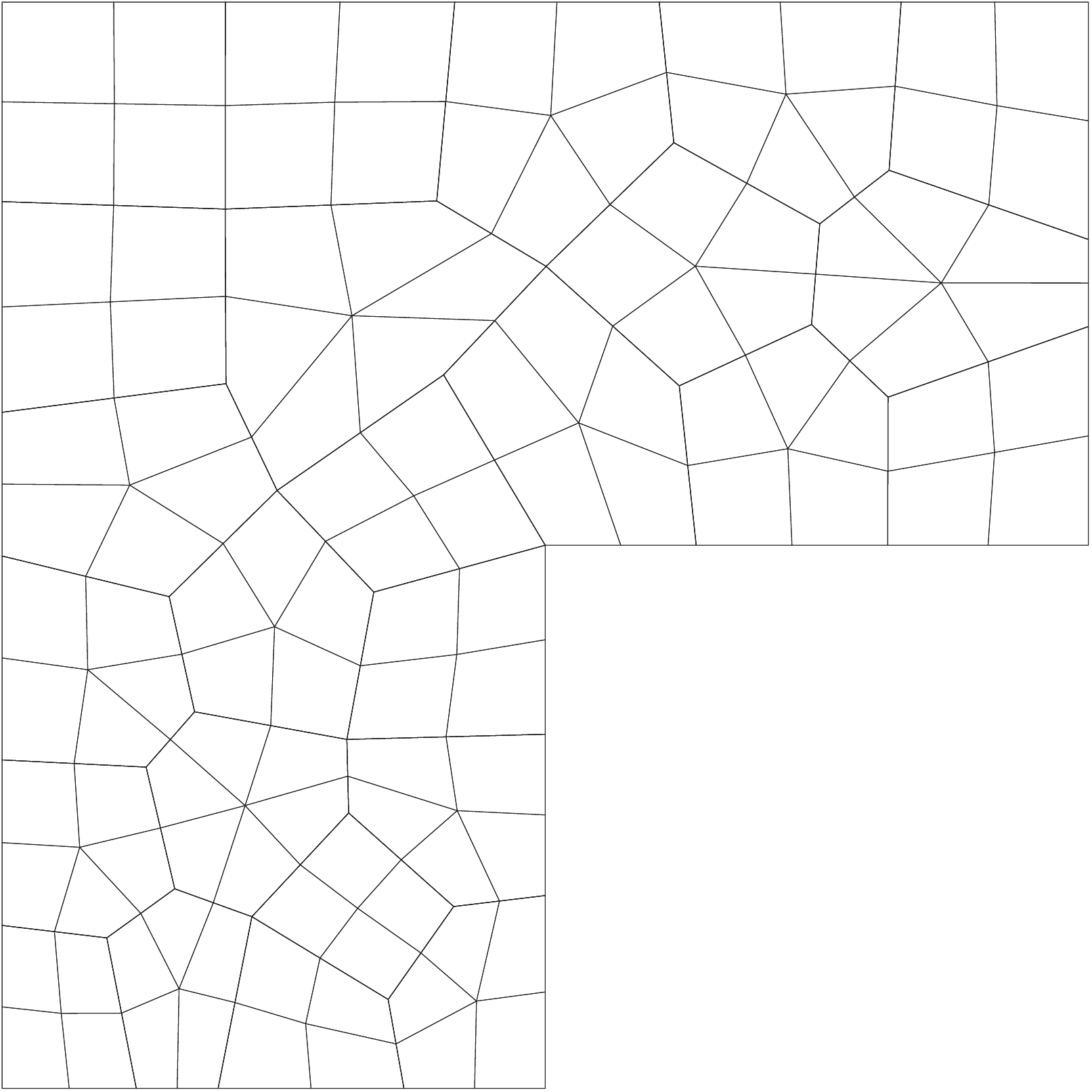}\qquad
 \includegraphics[width=1.2in]{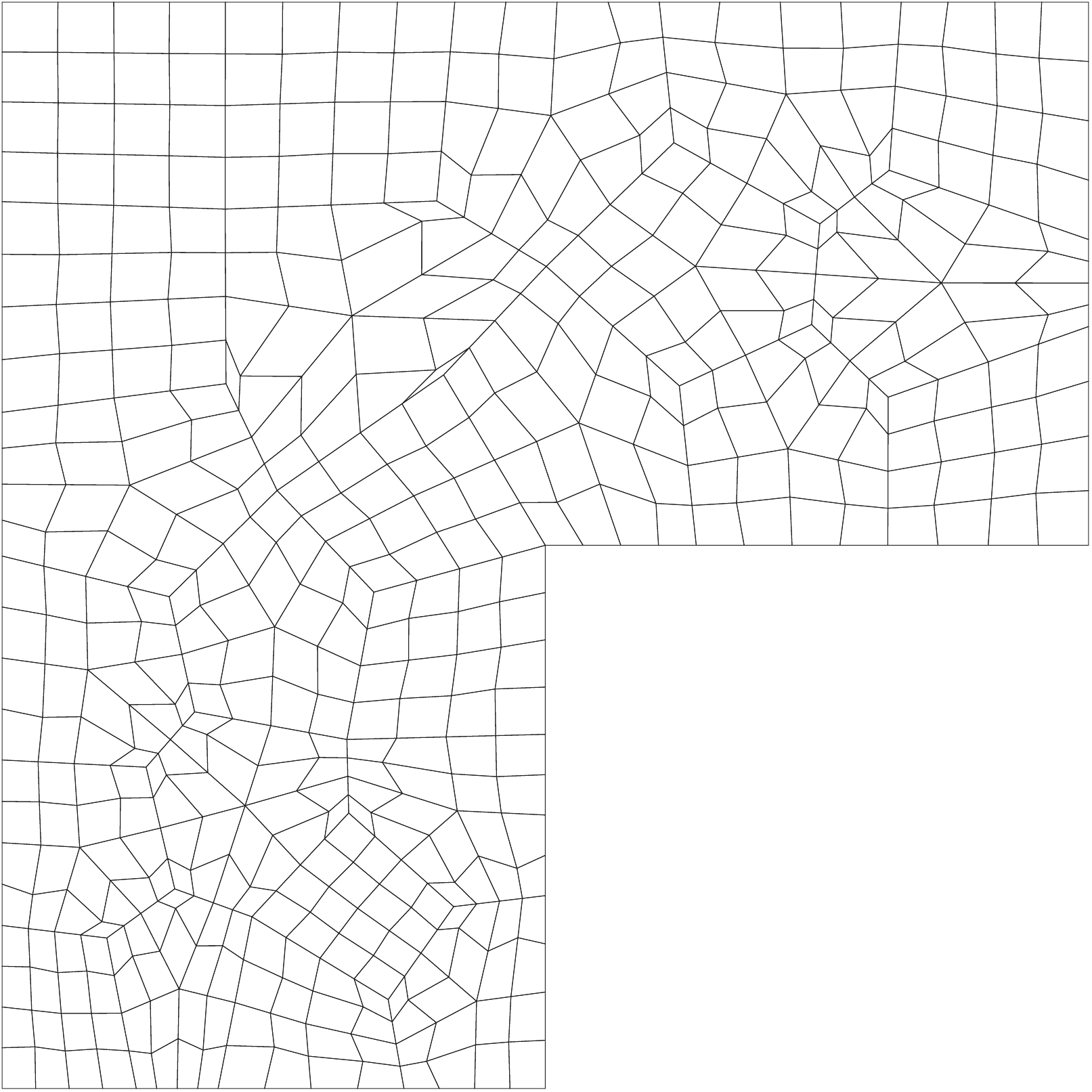}
\caption{Nested meshes on $\O=(-1,1)^2\setminus\big([0,1]\times[-1,0]\big)$.}
\label{fig:Gamma}
\end{figure}
\par
 The numerical results  for the virtual element methods are presented in
 Table~\ref{table:Gamma}, where the $O(h^\frac23)$ errors predicted by
 Theorem~\ref{thm:xihError} and Theorem~\ref{thm:buhError} (with $\omega=3\pi/2$)
 for $\|\bu-\bu_h\|_\LT$
 and $|\curl\bu-\PiO\xi_h|_{h,1}$ are observed.
 Note that the
 errors for the quadratic virtual element method are consistently smaller than the
 errors for the linear virtual element method with the same number of degrees of freedom.
\par
 The $O(h^\frac16)$ estimate for $\|\bn\times\bu_h\|_{L^2(\p\O)}$ predicted by
 Theorem~\ref{thm:buhError} is also observed.

\begin{longtable}{@{}|l|l|l|l|l|l|l|l|@{}}
  \caption{Errors and convergence rates for Experiment~3
   on nested meshes.}
  \label{table:Gamma}\\
  \hline
  $\qquad h$ & $\hspace{6pt}\#{\texttt{dofs}}$ & $\quad \texttt{rel}\;e_{\bu_h}$ & $\;\; \texttt{rate}$ & $\quad \texttt{rel}\;e_{\xi_h}$ & $\;\; \texttt{rate}$ & $\qquad \eBdry$ & $\;\; \texttt{rate}$ \\ \hline
  \multicolumn{8}{|c|}{$k = 1$} \\ \hline
  6.4474e-01 & 47     & -          & -       & -          & -       & - & -       \\
  3.6894e-01 & 143    & 2.8572e-01 & -       & 5.2525e-01 & -       & 8.9678e-03 & -       \\
  2.6844e-01 & 521    & 1.2388e-01 & 2.6282  & 2.5994e-01 & 2.2121  & 8.2326e-03 & 0.2690  \\
  1.3422e-01 & 1985   & 6.9851e-02 & 0.8266  & 1.3922e-01 & 0.9008  & 7.2007e-03 & 0.1932  \\
  6.7111e-02 & 7745   & 4.0430e-02 & 0.7888  & 7.0518e-02 & 0.9813  & 6.3781e-03 & 0.1750  \\
  3.3555e-02 & 30593  & 2.4014e-02 & 0.7515  & 3.5416e-02 & 0.9936  & 5.6825e-03 & 0.1666  \\
  1.6778e-02 & 121601 & 1.4529e-02 & 0.7249  & 1.7748e-02 & 0.9967  & 5.0699e-03 & 0.1646  \\
  8.3888e-03 & 484865 & 8.9082e-03 & 0.7058  & 8.8922e-03 & 0.9971  & 4.5234e-03 & 0.1645  \\ \hline

  \multicolumn{8}{|c|}{$k = 2$} \\ \hline
  6.4474e-01 & 143     & -          & -       & -          & -       & - & -       \\
  3.6894e-01 & 521     & 7.6210e-02 & -       & 1.5155e-01 & -       & 5.6603e-03 & -       \\
  2.6844e-01 & 1985    & 4.2800e-02 & 1.8144  & 4.7223e-02 & 3.6669  & 5.1152e-03 & 0.3184  \\
  1.3422e-01 & 7745    & 2.5322e-02 & 0.7572  & 1.4891e-02 & 1.6650  & 4.5512e-03 & 0.1685  \\
  6.7111e-02 & 30593   & 1.5784e-02 & 0.6820  & 4.5048e-03 & 1.7249  & 4.0639e-03 & 0.1634  \\
  3.3555e-02 & 121601  & 9.9197e-03 & 0.6701  & 1.5165e-03 & 1.5708  & 3.6267e-03 & 0.1642  \\
  1.6778e-02 & 484865  & 6.2467e-03 & 0.6672  & 6.7392e-04 & 1.1701  & 3.2347e-03 & 0.1650  \\
  8.3888e-03 & 1936385 & 3.9352e-03 & 0.6667  & 3.7998e-04 & 0.8267  & 2.8839e-03 & 0.1656  \\ \hline
\end{longtable}
%
\par\noindent
{\bf Experiment~4} \qquad We solve \eqref{eq:QuadCurl} on the domain
$\O=(0,1)^2\setminus [1/4,3/4]^2$ with $\beta=\gamma=1$,
 a smooth right-hand side
\begin{equation}\label{eq:SmoothRHS}
  \bof=\begin{bmatrix}
    (x_1^2+1)\sin x_1+x_1x_2^3+2\\
    (x_2^2+1)\cos x_1+x_1^3x_2^2-1
  \end{bmatrix},
\end{equation}
  and we use
  nested meshes obtained by refining a randomly generated initial Voronoi mesh
(cf. Figure~\ref{fig:Betti1}).
\begin{figure}[H]
 \centering
 \includegraphics[width=1.2in]{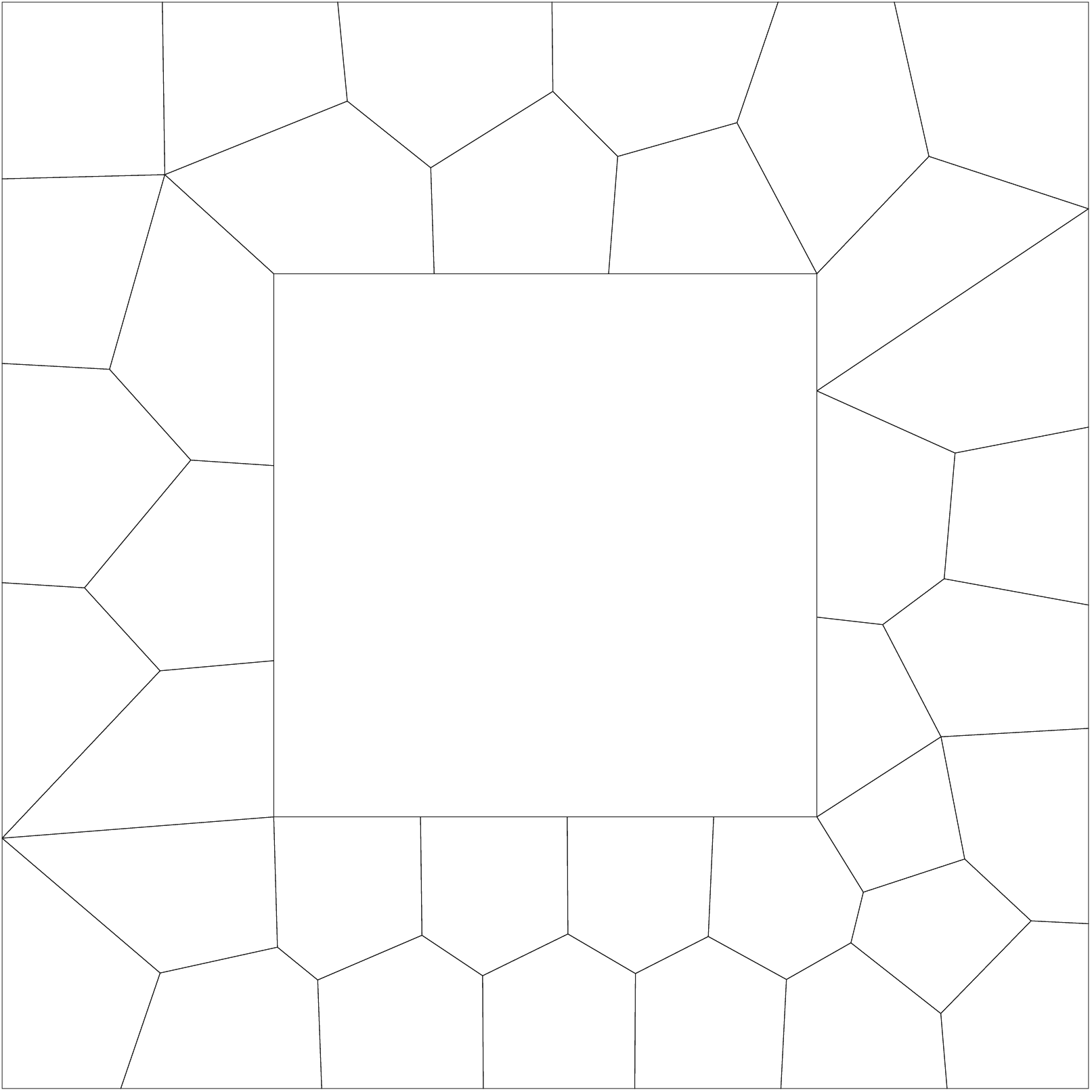}\qquad
 \includegraphics[width=1.2in]{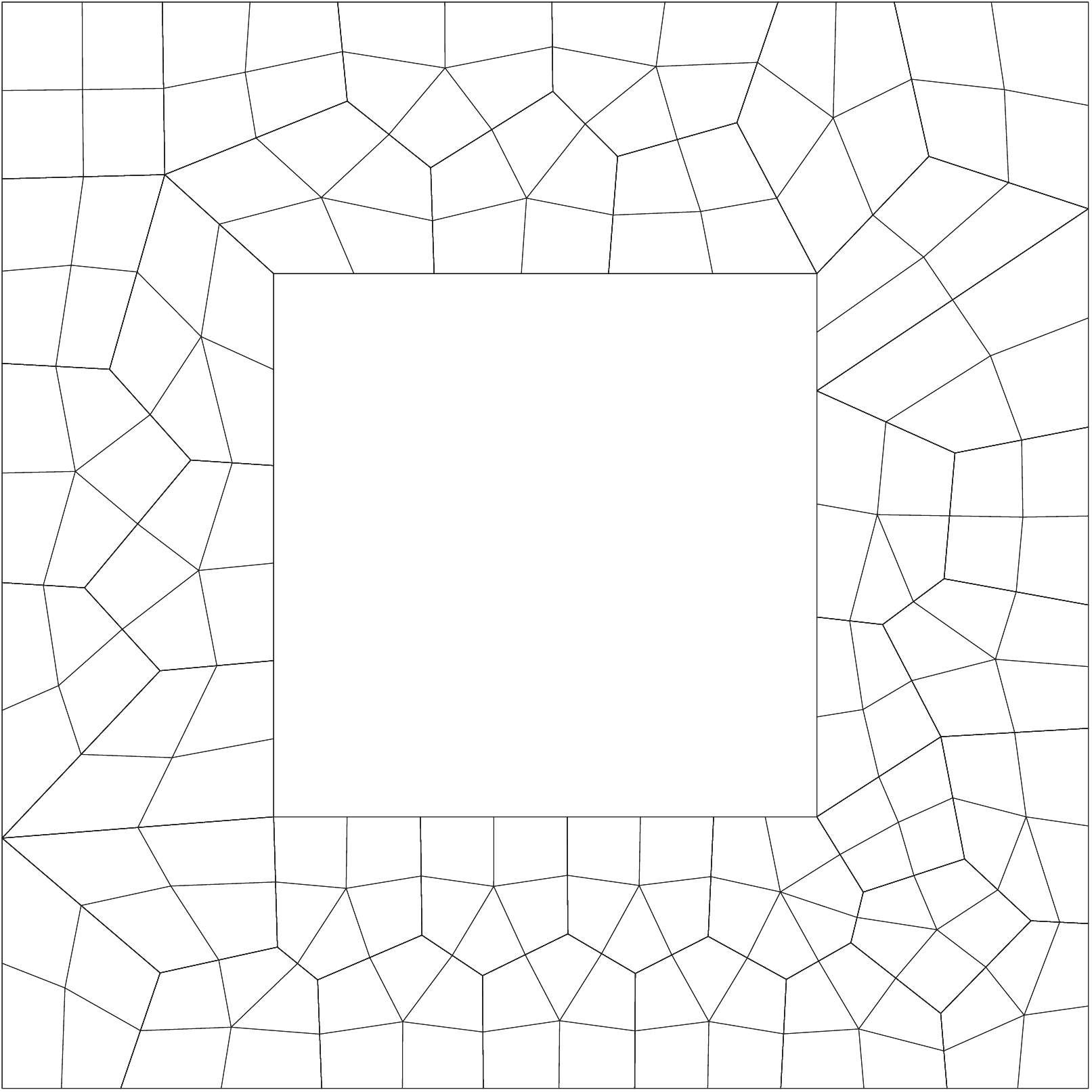}\qquad
 \includegraphics[width=1.2in]{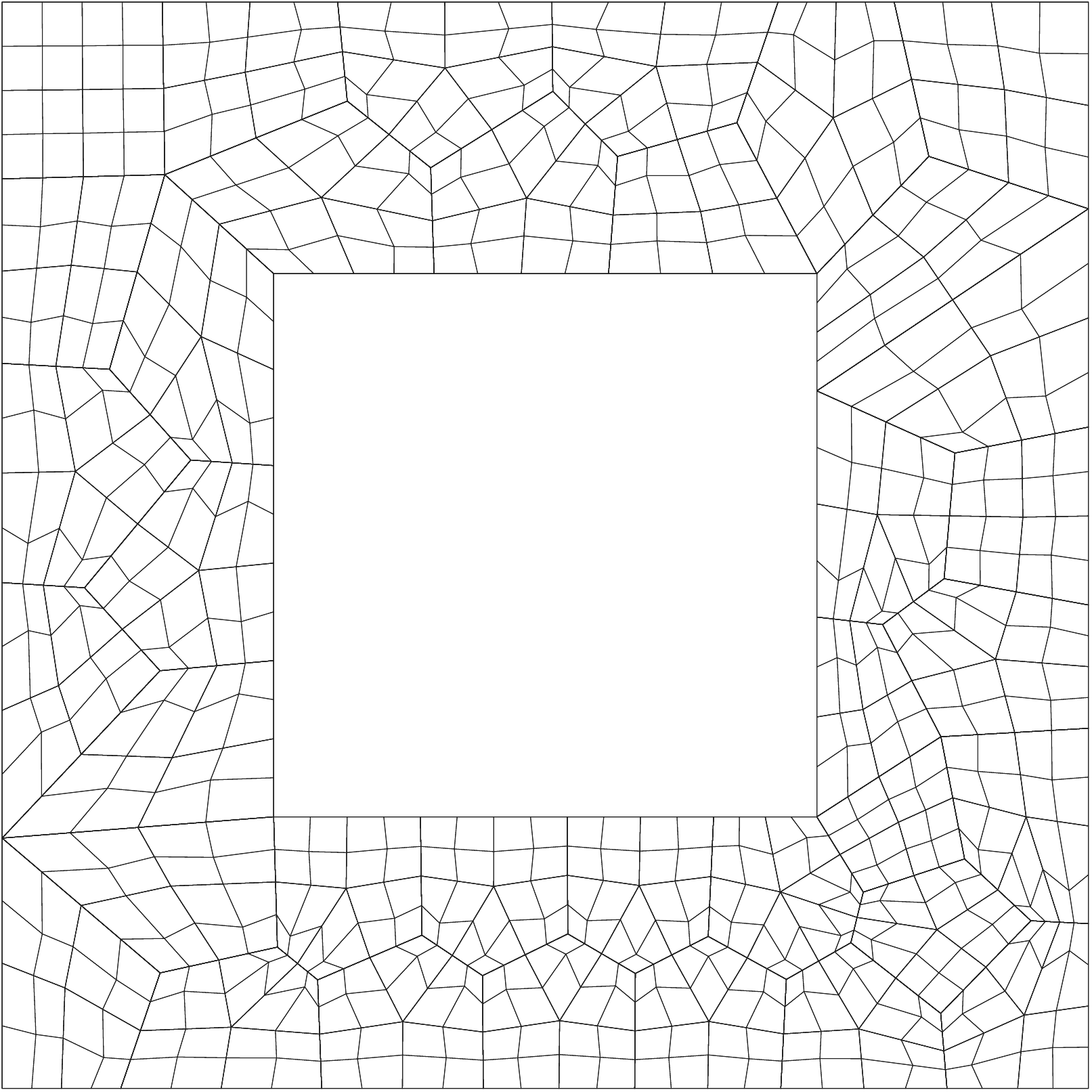}
\caption{Nested meshes on $\O=(0,1)^2\setminus [1/4,3/4]^2$.}
\label{fig:Betti1}
\end{figure}
\par
 The numerical results for the virtual element methods are displayed in
 Table~\ref{table:Betti1} where we also include the relative errors for the coefficient
 $c_1$ in the Hodge decomposition \eqref{eq:HodgeDecomposition} computed by
$$\texttt{rel} \; e_{c_1}^i = \frac{|c_{1,h_i} - c_{1,h_{i+1}}|}{|c_{1,h_{i+1}}|}.$$
\begin{table}[htbp]
  \caption{Errors and convergence rates for Experiment~4
   on nested meshes.}
  \label{table:Betti1}
    \resizebox{1 \textwidth}{!}{
\begin{tabular}{|c|c|c|c|c|c|c|c|c|c|c|}
  \hline
  $h$ & $\#{\texttt{dofs}}$ &
  $\texttt{rel}\;e_{\bu_h}$ & $\texttt{rate}$ &
  $\texttt{rel}\;e_{\xi_h}$ & $\texttt{rate}$ &
  $\eBdry$ & $\texttt{rate}$ &
  $c_1$ &
  $\texttt{rel}\;e_{c_1}$ & $\texttt{rate}$ \\ \hline

  \multicolumn{11}{|c|}{$k = 1$} \\ \hline
  3.0e-1 & 66      & -       & -     & -       & -     & -      & -      & -0.14937 & -       & -     \\
  1.7e-1 & 204     & 1.81e-1 & -     & 5.52e-1 & -     & 6.12e-4 & -      & -0.14961 & 1.58e-3 & -     \\
  1.1e-1 & 742     & 1.12e-1 & 1.08  & 2.60e-1 & 1.69  & 6.18e-4 & -0.022 & -0.14999 & 2.58e-3 & -1.10 \\
  5.5e-2 & 2820    & 7.02e-2 & 0.670 & 1.42e-1 & 0.875 & 5.55e-4 & 0.156  & -0.15019 & 1.33e-3 & 0.956 \\
  2.8e-2 & 10984   & 4.39e-2 & 0.676 & 7.67e-2 & 0.889 & 4.97e-4 & 0.159  & -0.15028 & 5.83e-4 & 1.19  \\
  1.4e-2 & 43344   & 2.75e-2 & 0.677 & 4.20e-2 & 0.869 & 4.45e-4 & 0.161  & -0.15032 & 2.42e-4 & 1.27  \\
  6.9e-3 & 1.72e5  & 1.72e-2 & 0.675 & 2.35e-2 & 0.836 & 3.97e-4 & 0.163  & -0.15033 & 9.85e-5 & 1.30  \\
  3.5e-3 & 6.86e5  & 1.08e-2 & 0.672 & 1.35e-2 & 0.800 & 3.54e-4 & 0.165  & -0.15034 & 3.96e-5 & 1.31  \\ \hline

  \multicolumn{11}{|c|}{$k = 2$} \\ \hline
  3.0e-1 & 204     & -       & -     & -       & -     & 0.0e+0  & -      & -0.14997 & -       & -     \\
  1.7e-1 & 742     & 8.72e-2 & -     & 1.45e-1 & -     & 4.25e-4 & -      & -0.15020  & 1.56e-3 & -     \\
  1.1e-1 & 2820    & 5.35e-2 & 1.10  & 6.61e-2 & 1.76  & 3.92e-4 & 0.179  & -0.15029 & 5.57e-4 & 2.31  \\
  5.5e-2 & 10984   & 3.24e-2 & 0.721 & 3.51e-2 & 0.911 & 3.51e-4 & 0.161  & -0.15032 & 2.27e-4 & 1.29  \\
  2.8e-2 & 43344   & 2.03e-2 & 0.676 & 2.11e-2 & 0.736 & 3.13e-4 & 0.165  & -0.15033 & 9.04e-5 & 1.33  \\
  1.4e-2 & 1.72e5  & 1.28e-2 & 0.670 & 1.31e-2 & 0.684 & 2.79e-4 & 0.165  & -0.15034 & 3.58e-5 & 1.34  \\
  6.9e-3 & 6.86e5  & 8.04e-3 & 0.667 & 8.25e-3 & 0.671 & 2.49e-4 & 0.166  & -0.15034 & 1.42e-5 & 1.34  \\
  3.5e-3 & 2.74e6  & 5.06e-3 & 0.667 & 5.19e-3 & 0.668 & 2.22e-4 & 0.166  & -0.15034 & 5.61e-6 & 1.33  \\ \hline
\end{tabular}
}
\end{table}
\par
 The $O(h^\frac23)$ errors for $\|\bu-\bu_h\|_\LT$ and $|\curl\bu-\PiO\xi_h|_{h,1}$
 predicted by Theorem~\ref{thm:xihError} and Theorem~\ref{thm:buhError} (with $\omega=3\pi/2$)
 are observed, and the quadratic virtual element method outperforms the linear
  virtual element method with the same number of degrees of freedom.
  The $O(h^\frac16)$ error for $\|\bn\times\bu_h\|_{L^2(\p\O)}$ also agrees with
 Theorem~\ref{thm:buhError}.
\par
 The error estimates for $|c_1-c_{1,h}|$ are better than the $O(h^\frac23)$ error predicted by
 Lemma~\ref{lem:CoefficientError}.  This is due to the fact that $\bof$ is a smooth vector
 field instead of just a vector field in $\LTT$ (cf. \cite[Remark~4.8 and Table~4.3]{BCNS:2012:Hodge}).
%
\begin{figure}[H]
 \centering
 \includegraphics[width=1.2in]{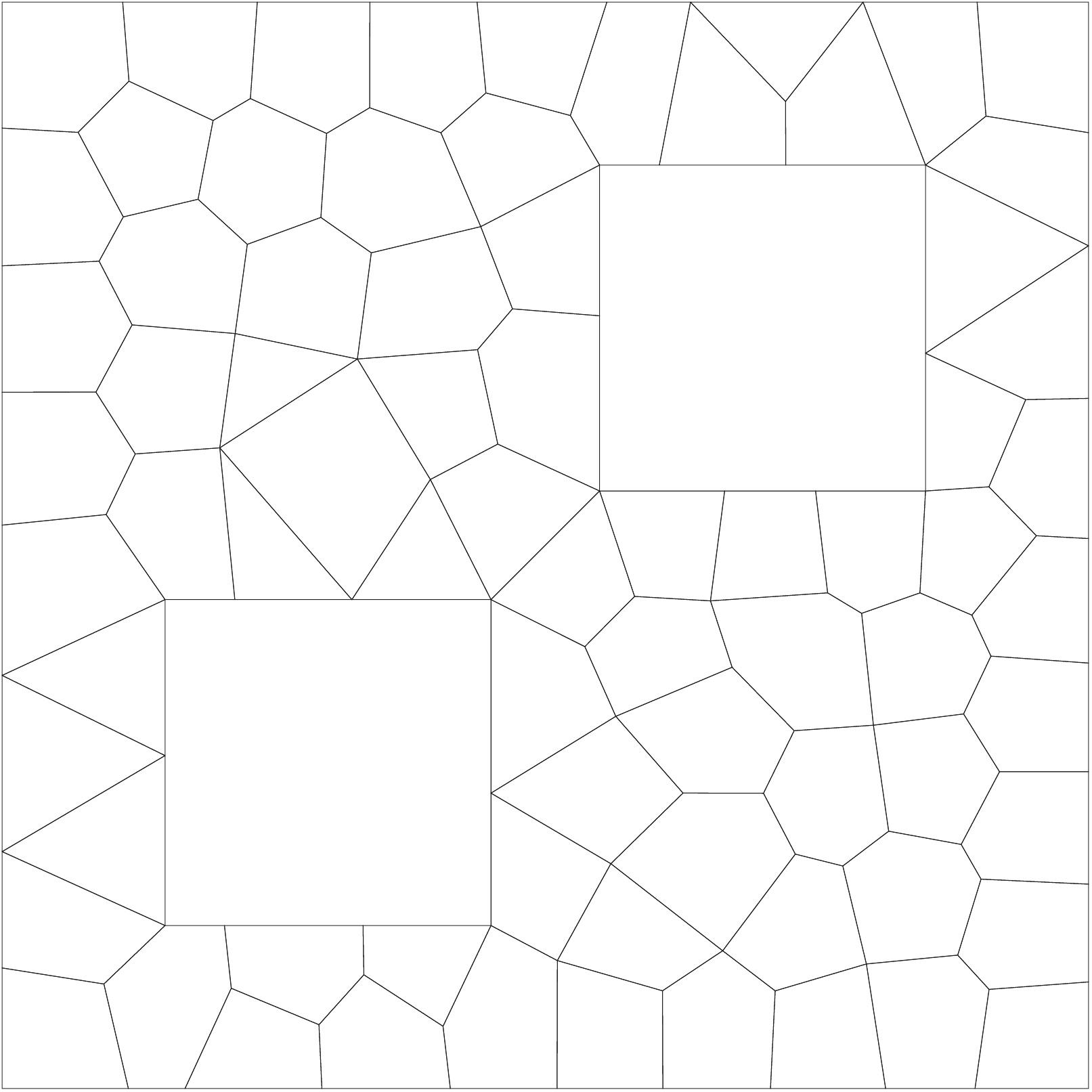}\qquad
 \includegraphics[width=1.2in]{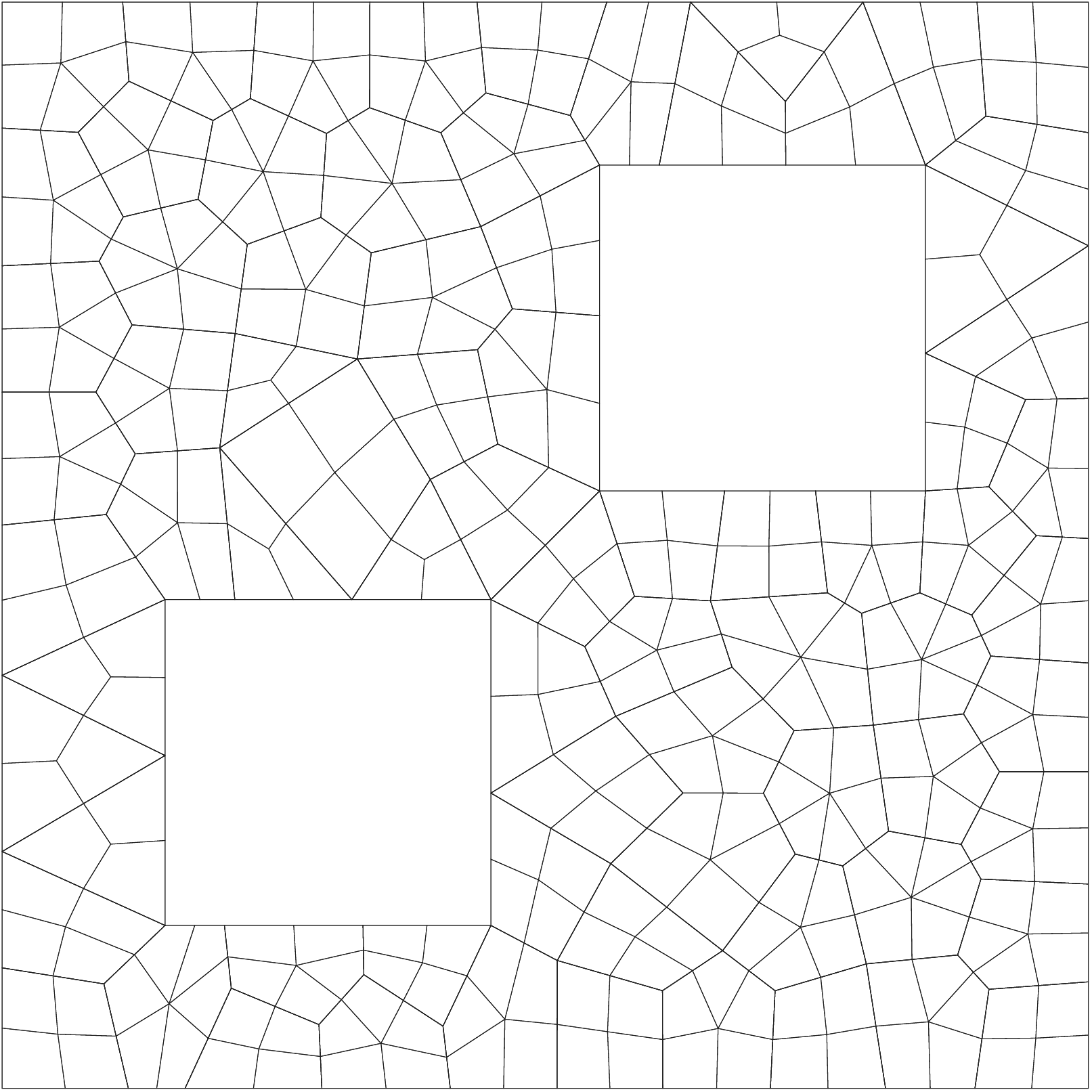}\qquad
 \includegraphics[width=1.2in]{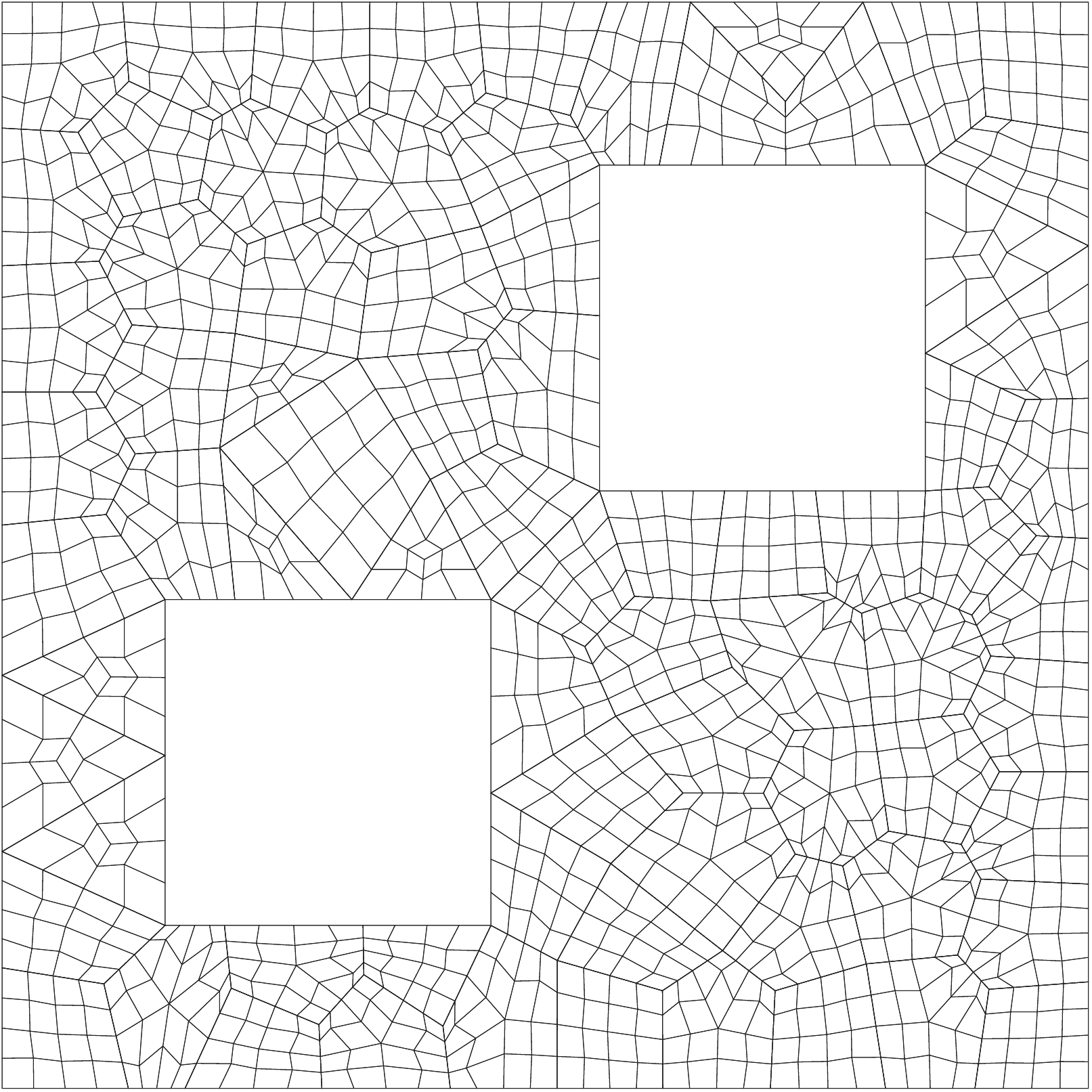}
\caption{Nested meshes on $\O=(0,1)^2\setminus\big([0.15,0.45]^2\times[0.55,0.85]^2\big)$.}
\label{fig:Betti2}
\end{figure}
\par\noindent
{\bf Experiment~5}\quad We solve \eqref{eq:QuadCurl} on the domain
$\O=(0,1)^2\setminus\big([0.15,0.45]^2\times[0.55,0.85]^2\big)$ with
$\beta=\gamma=1$, the smooth right-hand side $\bof$ in \eqref{eq:SmoothRHS},
and we use nested meshes obtained by refining a randomly generated initial Voronoi mesh
(cf. Figure~\ref{fig:Betti2}).

\par
 The numerical results for the virtual element methods are
 presented in Table~\ref{table:Betti2} together with the relative errors of the
 coefficients $c_1$ and $c_2$ in \eqref{eq:HodgeDecomposition}.
 They are similar to the results for Experiment~4.
{\small
\begin{table}[htbp]
\caption{Errors and convergence rates for Experiment~5 on nested meshes.}
  \label{table:Betti2}
  \resizebox{1 \textwidth}{!}{
\begin{tabular}{|l|l|l|l|l|l|l|l|l|l|l|l|l|l|}
  \hline
  $\quad h$ & $\#{\texttt{dofs}}$ &
  $\texttt{rel}\;e_{\bu_h}$ & $\texttt{rate}$ &
  $\texttt{rel}\;e_{\xi_h}$ & $\texttt{rate}$ &
  $\hspace{8pt}\eBdry$ & $\texttt{rate}$ &
  $\hspace{14pt} c_1$ & $\texttt{rel}\;e_{c_1}$ & $\texttt{rate}$ &
  $\quad c_2$ & $\texttt{rel}\;e_{c_2}$ & $\texttt{rate}$ \\ \hline
  \multicolumn{14}{|c|}{$k = 1$} \\ \hline
  2.2e-1 & 122    & -      & -      & -      & -      & - & -      & -0.08538 & -      & -      & -0.1559 & -      & -      \\
  1.2e-1 & 395    & 1.81e-1 & -     & 3.71e-1 & -     & 1.24e-3 & -     & -0.08523 & 1.7e-3 & -      & -0.1565 & 3.8e-3 & -      \\
  7.8e-2 & 1471   & 1.13e-1 & 1.12  & 1.75e-1 & 1.80  & 1.03e-3 & 0.45 & -0.08533 & 1.1e-3 & 1.01   & -0.1571 & 3.9e-3 & -0.05 \\
  3.9e-2 & 5663   & 7.22e-2 & 0.64 & 9.50e-2 & 0.88 & 8.73e-4 & 0.23 & -0.08540 & 8.8e-4 & 0.38  & -0.1574 & 2.0e-3 & 0.94  \\
  1.9e-2 & 22207  & 4.57e-2 & 0.66 & 5.09e-2 & 0.90 & 7.62e-4 & 0.19 & -0.08544 & 4.1e-4 & 1.10   & -0.1575 & 8.9e-4 & 1.18   \\
  9.7e-3 & 87935  & 2.88e-2 & 0.66 & 2.76e-2 & 0.88 & 6.73e-4 & 0.18 & -0.08545 & 1.7e-4 & 1.21   & -0.1576 & 3.7e-4 & 1.26   \\
  4.9e-3 & 3.5e5 & 1.81e-2 & 0.66 & 1.53e-2 & 0.85 & 5.97e-4 & 0.17 & -0.08546 & 7.3e-5 & 1.27   & -0.1576 & 1.5e-4 & 1.29   \\ \hline
  \multicolumn{14}{|c|}{$k = 2$} \\ \hline
  2.2e-1 & 395    & -      & -      & -      & -      & -& -      & -0.08540 & -      & -       & -0.1569 & -      & -      \\
  1.2e-1 & 1471   & 8.45e-2 & -     & 7.52e-2 & -     & 7.84e-4 & -     & -0.08542 & 2.6e-4 & -       & -0.1574 & 2.8e-3 & -      \\
  7.8e-2 & 5663   & 5.24e-2 & 1.14  & 3.67e-2 & 1.72  & 6.61e-4 & 0.40 & -0.08545 & 2.6e-4 & -3.3e-3 & -0.1575 & 1.0e-3 & 2.50   \\
  3.9e-2 & 22207  & 3.30e-2 & 0.67 & 2.04e-2 & 0.84 & 5.87e-4 & 0.17 & -0.08546 & 1.2e-4 & 1.06    & -0.1576 & 4.0e-4 & 1.33   \\
  1.9e-2 & 87935  & 2.06e-2 & 0.67 & 1.25e-2 & 0.71 & 5.23e-4 & 0.16 & -0.08546 & 4.9e-5 & 1.34    & -0.1576 & 1.5e-4 & 1.36   \\
  9.7e-3 & 3.5e5 & 1.30e-2 & 0.67 & 7.80e-3 & 0.67 & 4.66e-4 & 0.16 & -0.08546 & 1.9e-5 & 1.34    & -0.1576 & 6.1e-5 & 1.35   \\
  4.9e-3 & 1.4e6 & 8.16e-3 & 0.66 & 4.91e-3 & 0.66 & 4.15e-4 & 0.16 & -0.08546 & 7.7e-6 & 1.33    & -0.1576 & 2.4e-5 & 1.34   \\ \hline
\end{tabular}
}
\end{table}
}
\goodbreak
\section{Concluding Remarks}\label{sec:Conclusions}
 We have designed two virtual element methods with order
 $k=1$ and $k=2$ for the quad-curl problem \eqref{eq:QuadCurl} based on the
 Hodge decomposition
 approach.  They
 enjoy the optimal convergence allowed by the regularity of the solution.
 The method
 with order $k=1$ is the simplest polytopal method to date for the quad-curl
 problem in two dimensions.
\par
 We note that finite element methods for the quad-curl problem in three
 dimensions based on the Hodge decomposition approach have been developed in
 \cite{BCS:2024:3DQuadCurl}.  Therefore it is possible to design
 polytopal methods in three dimensions
 based on the $H(\div)$ and $H(\bcurl)$ conforming virtual element methods in \cite{BBMR:2016:VEMCurlDiv}.
\par
 The Hodge decomposition approach has also been applied to other
 electromagnetics problems in
 \cite{AM:2010:Hodge,BCNS:2012:Hodge,BGS:2017:Impedance,BCS:2025:3DHodgeMaxwell}.
 One can therefore
 apply the ideas in this paper to develop
 virtual element methods for these problems.
\par
 Finally, other polytopal methods for the quad-curl problem can also be developed
 by combining the Hodge decomposition approach with methodologies such as weak Galerkin
 methods (cf. \cite{MWY:2015:PolytopalWG}), discontinuous Galerkin methods
 (cf. \cite{CDGH:2017:DG})
 and hybrid high-order methods (cf. \cite{DD:2020:HHOBook}).

\appendix

\section{Error Analysis of Virtual Element Methods}\label{append:VEMError}
 Let $V$ be either $\HOne$ or $\HOnez$,  and $B(\cdot,\cdot)$ be a
 symmetric  bilinear form
 on $V$ that is bounded and coercive.
\begin{remark}\label{rem:ContinuousProblems}
  For the Neumann boundary value problem in Section~\ref{subsec:phi},
  $V=\HOne$ and the bilinear form
  $B(\cdot,\cdot)$ is given by
  $$B(v,w)=(\bcurl v,\bcurl w)_\LT+(v,1)_\LT(w,1)_\LT.$$
 For the Dirichlet boundary value problem defined by \eqref{subeqs:HarmonicFunctions},
 $V=\HOnez$ and
 $$B(v,w)=(\grad v,\grad w)_\LT.$$
 For the Dirichlet boundary value problems in Section~\ref{subsubsec:xihEstimateZero}
 and Section~\ref{subsubsec:xihEstimatePositive},
 $V=\HOnez$ and
  $$B(v,w)=(\bcurl v,\bcurl w)_\LT+\beta(v,w)_\LT.$$
\end{remark}
 The continuous problem is to find $u\in V$ such that
\begin{equation}\label{eq:ContinuousProblem}
  B(u,v)=(g,v)_\LT\qquad\forall\,v\in V,
\end{equation}
 where $g$ belongs to $\HOne$.
\par
 Let $\cQ_h^k\subset V$ be a virtual element space of order
 $k$ associated with a polytopal
 mesh $\cT_h$ of $\O$ (cf. Section~\ref{subsec:VEM}).
 The discrete problem is to find $u_h\in \cQ_h^k$ such that
\begin{equation}\label{eq:DiscreteProblem}
  B_h(u_h,v_h)=(g,\PiZ v_h)_\LT\qquad\forall\,v_h\in \cQ_h^k.
\end{equation}
\par
 Here $B_h(\cdot,\cdot)$ is a symmetric bilinear form defined on $\cQ_h^k$.
 We assume that $B_h(\cdot,\cdot)$ has the following properties:
\begin{equation}\label{eq:BhTwo}
  B_h(v_h,v_h)\geq C_\heartsuit \|v_h\|_\HOne^2 \qquad
  \forall\,v_h\in\cQ_h^k,
\end{equation}
 and
\begin{align}\label{eq:BhThree}
  |B_h(v_h,w_h)-B(v_h,w_h)|
  &\leq C_\diamondsuit \big(|v_h-\PiO v_h|_{h,1}+\|v_h-\PiZ v_h\|_\LT\big)\|w_h\|_\HOne
\end{align}
 for all $v_h,w_h\in\cQ_h^k$, where the operator
 $\PiO:V\longrightarrow \Poly_{k}(\O;\cT_h)$ is induced by
 a local projection $\Pi_{k,D}^1:\HOneD\longrightarrow\Poly_{k}(D)$
 that satisfies the properties in Section~\ref{subsubsec:PiO}.
\begin{remark}\label{rem:DiscreteProblems}
  For the discrete problems in Section~\ref{subsubsec:xihEstimateZero}
  and Section~\ref{subsubsec:xihEstimatePositive},
  $\cQ_h^k=\zV_h^k\subset\HOnez$ and
  the bilinear form $B_h(\cdot,\cdot)$ is given by
 $$B_h(v_h,w_h)=a_h(v_h,w_h)+\beta(\PiZ v_h,\PiZ w_h)_\LT,$$
 where $a_h(\cdot,\cdot)$ is defined by \eqref{eq:ah}.
  For the discrete problem \eqref{eq:phih},
  $\cQ_h^k=V_h^k\subset\HOne$ and we have
 $$B_h(v_h,w_h)=a_h(v_h,w_h)+(v_h,1)_\LT(w_h,1)_\LT.$$
 For the discrete problems in \eqref{subeqs:DHarmonicFunctions},
 $\cQ_h^k=\zV_h^k\subset\HOnez$
  and $B_h(\cdot,\cdot)$ is given by
  $$B_h(v_h,w_h)=a_h(v_h,w_h).$$
 One can use the properties of virtual elements in Section~\ref{subsubsec:PiO},
 Section~\ref{subsubsec:SD} and Section~\ref{subsubsec:Inconsistency}
 to verify
 \eqref{eq:BhTwo} and \eqref{eq:BhThree} in all these cases.
\end{remark}
\par
 We also assume that there is an interpolation operator
 $I_{k,h}:H^{1+s}(\O)\longrightarrow \cQ_h^k$
 defined for $s>0$ whose restriction to a polygon $D\in\cT_h$ satisfies the estimates in
 Section~\ref{subsubsec:Interpolation}.
\par
 We have, from \eqref{eq:BhTwo},
\begin{align}\label{eq:AbstractError}
  \|u-u_h\|_\HOne&\leq \|u-I_{k,h}u\|_\HOne+\|I_{k,h}u-u_h\|_\HOne\\
  &\leq \|u-I_{k,h}u\|_\HOne+C_\heartsuit^{-1}\max_{w_h\in \cQ_h^k}
  \frac{B_h(I_{k,h}u-u_h,w_h)}{\|w\|_\HOne}\notag
\end{align}
 and, in view of \eqref{eq:ContinuousProblem} and
 \eqref{eq:DiscreteProblem}, we can write
\begin{align}\label{eq:Numerator}
  B_h(I_{k,h}u-u_h,w_h)
 &=B_h(I_{k,h}u,w_h)-B(I_{k,h}u,w_h)+B(I_{k,h}u,w_h)-B(u,w_h)\\
 &\hspace{40pt}+(g,w_h-\PiZ w_h)_\LT.\notag
\end{align}
\par
 It follows from \eqref{eq:BhThree} that
\begin{align}\label{eq:DifferenceBdd}
  &|B_h(I_{k,h}u,w_h)-B(I_{k,h}u,w_h)|\\
  &\hspace{50pt}\leq
   C_\diamondsuit \big(|I_{k,h}u-\PiO I_{k,h}u|_{h,1}+
   \|I_{k,h}u-\PiZ I_{k,h}u\|_\LT\big)\|w_h\|_\HOne,
   \notag
\end{align}
 and, by the boundedness of $B(\cdot,\cdot)$,
\begin{equation}
  |B(I_{k,h}u,w_h)-B(u,w_h)|\leq C\|I_{k,h}u-u\|_\HOne\|w_h\|_\HOne.
\end{equation}
 Finally, by a standard error estimate for $\PiZ$, we have
\begin{equation}\label{eq:gError}
  (g,w_h-\PiZ w_h)_\LT=(g-\PiZ g,w_h-\PiZ w_h)_\LT\leq C h^2\|g\|_\HOne\|w_h\|_\HOne.
\end{equation}
\par
 Putting \eqref{eq:AbstractError}--\eqref{eq:gError} together, we arrive at
 the estimate
\begin{align*}
  \|u-u_h\|_\HOne&\leq C\Big[\|u-I_{k,h}u\|_\HOne+ |I_{k,h}u-\PiO I_{k,h}u|_{h,1}\\
  &\hspace{60pt}+
  \|I_{k,h}u-\PiZ I_{k,h}u\|_\LT+h^2\|g\|_\HOne\Big],
  \notag
\end{align*}
 which together with \eqref{eq:InterpolationOne} and \eqref{eq:InterpolationZero}
 implies
\begin{equation}\label{eq:ConcreteError}
  \|u-u_h\|_\HOne\leq \CExpo \hfactor \|g\|_\HOne,
\end{equation}
 provided that $k=1,2$ and $u$ satisfies
\begin{equation}\label{eq:EllipticRegularity}
  \|u\|_\RegStar\leq \CExpo \|g\|_\HOne.
\end{equation}
\par
 Finally  \eqref{eq:PiOPythagoras}, \eqref{eq:PiOErrorEstimate},
 \eqref{eq:ConcreteError} and \eqref{eq:EllipticRegularity} imply
\begin{align*}
  |u-\PiO u_h|_{h,1}&=|(u-\PiO u)+\PiO (u-u_h)|_{h,1}\notag\\
                    &\leq |u-\PiO u|_{h,1}+|u-u_h|_\HOne\\
                    &\leq \CExpo \hfactor \|g\|_\HOne.\notag
\end{align*}
\begin{remark}\label{rem:Limit}
  The $h^2$  that appears in \eqref{eq:gError} is the reason
  that we only consider $k=1$ and $k=2$.
\end{remark}

\end{document}